\documentclass[secthm,seceqn,amsthm,ussrhead,10pt]{amsart}
\usepackage{amsmath,latexsym}
\usepackage[english]{babel}
\usepackage[psamsfonts]{amssymb}
\usepackage{times}
\usepackage{cite}
\usepackage{pdflscape} 
\usepackage{ulem}
\usepackage[mathcal]{euscript}
\usepackage{tikz}
\usepackage{hyperref}
\usepackage{stmaryrd}
\usetikzlibrary{arrows}

\newcommand{\Co}{{\mathbb C}}

\newtheorem{Th}{Theorem}
\newtheorem{Lem}[Th]{Lemma}

\newtheorem{Remark}[Th]{Remark}

\newenvironment{Proof}[1][Proof.]{\begin{trivlist}
\item[\hskip \labelsep {\bfseries #1}]}{\flushright
$\Box$\end{trivlist}}

\begin{document}

{\Large On a ternary generalization of Jordan algebras
\footnote{The work was supported by RFBR 17-01-00258 }

}
\medskip

\medskip

\medskip

\medskip

\textbf{Ivan Kaygorodov$^a$, Alexander Pozhidaev$^{b,c}$, Paulo Saraiva$^d$} 
\medskip

{\tiny

$^a$ Universidade Federal do ABC, Santo Andre, Brasil

$^b$ Novosibirsk State University, Russia

$^c$ Sobolev Institute of mathematics, Novosibirsk, Russia

$^d$ Universidade de Coimbra, CeBER, CMUC e FEUC, Coimbra, Portugal
\smallskip

  E-mail addresses:\smallskip

  Ivan Kaygorodov (kaygorodov.ivan@gmail.com),

  Alexandre Pozhidaev (app@math.nsc.ru),

  Paulo Saraiva (psaraiva@fe.uc.pt).

 }

\

\noindent \textbf{Abstract:} Based on the relation between the notions of Lie triple systems and Jordan algebras, we introduce the $n$-ary Jordan algebras,
an $n$-ary generalization of Jordan algebras obtained via the generalization of the following property $\left[ R_{x},R_{y}\right] \in Der\left( \mathcal{A}\right),$
where $\mathcal{A}$ is an $n$-ary algebra. Next, we study a ternary example of these algebras. Finally, based on the construction of a family of ternary algebras defined by means of the Cayley-Dickson algebras, we present an example of a ternary $D_{x,y}$-derivation algebra ($n$-ary $D_{x,y}$-derivation algebras are the non-commutative version of $n$-ary Jordan algebras).

\


\

\section{Introduction}

\hspace{0cm} The notion of Jordan algebra appeared in $1934$  as
the underlying algebraic structure for certain operators in quantum
mechanics \cite{Jordan_et_al}.  Recall that a Jordan algebra is a commutative algebra over a field $\mathbb{F}$
($char\left(\mathbb{F}\right)\neq2$) satisfying the so-called Jordan identity:
\begin{equation}
\left( xy\right) x^{2}=x\left( yx^{2}\right) .\hspace{2cm}
\label{Jordan}
\end{equation}
Since then, the theory of Jordan algebras has been developed, not only in purely algebraic aspects, but also intertwined with other subjects and applications. For instance, the vast class of noncommutative Jordan algebras (it includes, \textit{e.g.}, alternative algebras, Jordan algebras, quasiassociative algebras, quadratic flexible algebras and anticommutative algebras) attracted a lot of attention. Schafer proved that a simple noncommutative Jordan algebra is either a simple Jordan algebra, or a simple quasiassociative algebra, or a simple flexible algebra of degree $2$ \cite{Schafer_non}. Concerning the intervention of Jordan algebras in other areas, and just to mention a couple of these, we can find applications in differential geometry (see  \cite{Faraut_Koranyi} and \cite{Upmeier}) and in optimization methods (see \cite{Faybusovich_2010}). Further, for a motivation and a general overview of Jordan algebras (including applications), read \cite{McCrimmon} and \cite{Iordanescu}.

A related issue has been the attempt to generalize the Jordan algebra structure to the case of algebras with $n$-ary multiplication, with an emphasis to the ternary case. Mostly, these generalizations include Jordan triple systems (as in \cite{bremner_hentzel_2000} and \cite
{gnedbaye2007}), but also other ternary versions (\textit{e.g}, \cite{Bremner2001}). In the present paper we follow a different approach.

According to \cite{osborn69} and \cite{Sidorov_1981}, a  Lie triple algebra is a commutative, nonassociative algebra $\mathcal{A}$  over a field $\mathbb{F}$
($char\left(\mathbb{F}\right)\neq2$) satisfying
\begin{equation}
\left( a,b^{2},c\right)=2b\left( a,b,c\right), 
 \label{Osborn_def}
\end{equation}
where $\left( .,.,.\right)$ stands for the  associator,
\begin{equation*}
\left( a,b,c\right)=(ab)c-a(bc).
\end{equation*}
It is not difficult to check that this identity is equivalent to:
\begin{equation}
R_{\left( x,y,z\right) }=\left[ R_{y},\left[ R_{x},R_{z}\right]
\right], \label{Osborn}
\end{equation}
where $\left[.,.\right]$ stands for the  commutator,
\begin{equation*}
\left[a,b\right]=ab-ba,
\end{equation*}
and $R_{x}$ is a  right multiplication operator,  \textit{i.e.}
\begin{equation*}
y\mapsto yR_{x}=yx .
\end{equation*}

It is simple to observe that every Jordan algebra is a Lie triple algebra, although the opposite is not necessarily true. 
Further, on a commutative algebra, $\mathcal{A}$, the identity (\ref{Osborn}) is equivalent to
\begin{equation}
\left[ R_{x},R_{y}\right] \in Der\left(\mathcal{A}\right) ,
\label{D_x,y-identity}
\end{equation}
where $Der\left( \mathcal{A}\right)$ stands for the Lie algebra of derivations of $\mathcal{A}$. Writing $D_{x,y}$ instead of $\left[ R_{x},R_{y}\right]$, this means that 
\begin{equation}
D_{x,y}\left(ab \right)= D_{x,y}(a)b+a D_{x,y} (b).
\end{equation}
It is also known that,
in a Jordan algebra $\mathcal{J}$,
\begin{equation*}
Inder\left( \mathcal{J}\right) =\left\{ \sum \left[ R_{x_{i}},R_{y_{i}}\right]
:x_{i},y_{i}\in \mathcal{J}\right\}  ,
\end{equation*}
where $Inder\left( \mathcal{J}\right)$ stands for the Lie algebra of inner derivations of $\mathcal{J}$. Thus, in every Jordan algebra $\mathcal{A}$, the commutator of two arbitrary right multiplication operators is a derivation (an inner derivation, to be precise) of $\mathcal{A}$ (see \cite{Faulkner_1967}).

Let $\mathcal{A}$ be an $n$-ary algebra with a multilinear multiplication $\llbracket .,\ldots ,.\rrbracket :\times^{n}\mathbb{V}\to \mathbb{V}$, where $\mathbb{V}$ is the underlying vector space. We propose the following definition: $\mathcal{A}$ is said to be an \textbf{$n$-ary Jordan algebra} if
\begin{equation}
\llbracket x_{\sigma \left( 1\right) },\ldots ,x_{\sigma \left( n\right)
}\rrbracket =\llbracket x_{1},\ldots ,x_{n}\rrbracket  \label{tot-com}
\end{equation}
for every permutation $\sigma \in \mathcal{S}_{n}$ and for every $%
x_{2},\ldots ,x_{n},y_{2},\ldots ,y_{n}\in \mathbb{V}$ and if
\begin{equation}
\left[ R_{(x_2,\ldots ,x_n)},R_{(y_2,\ldots ,y_{n})}\right] \in Der\left(
\mathcal{A}\right) ,  \label{der-Jordan}
\end{equation}
where $\left[ .,.\right] $ stands again for the commutator and $R_{(x_{2},\ldots
,x_{n})},R_{(y_{2},\ldots ,y_{n})}$ are the right multiplication operators, defined in
the usual way:
\begin{equation*}
y\mapsto yR_{(x_{2},\ldots
,x_{n})}=\llbracket y,x_{2},\ldots ,x_{n}\rrbracket  .
\end{equation*} 
For the sake of simplicity, we will often write $R_{x}$ instead of
$R_{(x_{2},\ldots ,x_{n})}$  and, analogously to the binary case, $D_{x,y}$ instead of $\left[R_{x},R_{y}\right] $. Further, we will call a $D_{x,y}$-derivation algebra to every $n$-ary with the identity (\ref{der-Jordan}). Under this notation, (\ref{der-Jordan}) can be written in the following way:
\begin{equation}
 D_{x,y}\ \llbracket z_{1},...,z_{n} \rrbracket = 
 \sum_{i=1}^{n} \llbracket z_{1},..., D_{x,y}\left(z_{i}\right),...,z_{n}\rrbracket. \label{der-Jordan_Dxy}
\end{equation}
Throughout this paper, (\ref{tot-com}) is the total commutativity identity and (\ref{der-Jordan}) (or, equivalently (\ref{der-Jordan_Dxy})) will be cited as the $ D_{x,y}$-identity.

The paper is organized in the following way. In the second section we consider some ternary algebras defined on the direct sum of a field and a vector space, by defining  a general mutiplication depending on three given forms. Discussing the possible cases for these forms, we obtain the first examples of ternary Jordan algebras. 

The third section is devoted to a particular case of the general ternary mutliplication defined in the previous section, restricted to a vector space (over a field of characteristic zero). It turns out that this provides  a new example of ternary Jordan algebra, denoted by $\mathbb{A}$, which is simple. We study its identities of degrees $1$ and $2$ concluding that these result from the total commutativity property. Finally, we conclude that the proposed notion of ternary Jordan algebra doesn't coincide with the notion of Jordan triple system. 

In the fourth section we study the derivation algebra of the simple ternary Jordan algebra $\mathbb{A}$ introduced in the previous section, concluding that it coincides with $so(n)$ and that all derivations of $\mathbb{A}$ are inner.

The fifth section is focused on the search of new examples of ternary Jordan algebras. There, dealing with matrix algebras, we obtain two non-isomorphic symmetrized matrix subalgebras, one of which is simple. Further, defining a certain ternary multiplication on the algebras obtained by the Cayley-Dickson doubling process, we define a $4$-dimensional $D_{x,y}$-algebra over the generalized quaternions. Finally, we present an analog of the TKK-construction to ternary algebras, obtaining new examples of ternary Jordan algebras.

In the last section we recall the concept of reduced algebras of $n$-ary algebras. After this, we conclude that, oppositely to other classes of algebras, the reduced algebras of the ternary Jordan algebra $\mathbb{A}$ are not Jordan algebras in general. We note that other generalizations of Jordan algebras (\textit{e.g.}, Jordan triple systems) also fail this property.

\section{Ternary algebras with a generalized multiplication}

\label{ternary}

Let us consider an $n$-dimensional vector space $\mathbb{V}$ over a field $\mathbb{F}$
equipped with two bilinear, symmetric and nondegenerate forms, $f$ and $h,$
and also with a trilinear, symmetric and nondegenerate form $g$.
Given a basis $\mathcal{B}=\left\{ b_{1},\ldots ,b_{n}\right\} $ of $\mathbb{V}$, those forms are such that:
\begin{equation}
f\left( b_{i},b_{j}\right) =\delta _{ij},\qquad h\left( b_{i},b_{j}\right)
=\delta _{ij}\quad \text{and}\quad g\left( b_{i},b_{j},b_{k}\right) =\delta
_{ijk},
\end{equation}
where $\delta _{ij}$ and $\delta _{ijk}$ are Kronecker deltas.

Consider now a binary multiplication $*$ on the vector space  $\mathbb{F} \oplus \mathbb{V}$
defined by 
\begin{equation*}
\left( \alpha +u \right) * \left(\beta +v \right) = \alpha\beta +f(u,v)+ \alpha v +\beta u, \alpha,\beta \in \mathbb{F}, u,v \in \mathbb{V}.
\end{equation*}

 Then we obtain a Jordan algebra of a symmetric bilinear form $f$, denoted by $J(\mathbb{V},f)$, which is simple if $\dim \mathbb{V} >1$ and $f$ is nondegenerate.
 
 Seeking an analogue of $J(\mathbb{V},f)$ in the case of ternary algebras, we will consider
the same vector space $\mathbb{F} \oplus \mathbb{V}$, where we define a (most general) trilinear multiplication $\llbracket .,.,.\rrbracket$, such that 
 
\begin{equation}\label{general-mult}\llbracket \alpha_{1}+v_{1},\alpha _{2}+v_{2},\alpha _{3}+v_{3}\rrbracket = \end{equation}
$$\label{gen_mult}=\left(\alpha _{1}\alpha
_{2}\alpha _{3}+ \alpha _{1}f\left( v_{2},v_{3}\right) +\alpha _{2}f\left(
v_{1},v_{3}\right) +\alpha _{3}f\left( v_{1},v_{2}\right) +g\left( v_{1},v_{2},v_{3}\right) \right)+ $$
$$+\left( \alpha _{2}\alpha _{3}+h\left( v_{2},v_{3}\right) \right)
v_{1}+\left( \alpha _{1}\alpha _{3}+h\left( v_{1},v_{3}\right) \right)
v_{2}+\left( \alpha _{1}\alpha _{2}+h\left( v_{1},v_{2}\right) \right) v_{3},$$
for arbitrary $\alpha_{i}\in \mathbb{F}$ and $v_{i}\in \mathbb{V}$. The obtained ternary algebra will be denoted by $\mathcal{V}_{f,g,h}$.

Under this assumption, it is clear that $\llbracket .,.,.\rrbracket $ is totally
commutative, that is:
\begin{equation}
\llbracket \alpha_{\sigma \left( 1\right) }+v_{\sigma \left( 1\right)
},\alpha_{\sigma \left( 2\right) }+v_{\sigma \left( 2\right) },\alpha
_{\sigma \left( 3\right) }+v_{\sigma \left( 3\right) }\rrbracket =\llbracket
\alpha _{1}+v_{1},\alpha _{2}+v_{2},\alpha _{3}+v_{3}\rrbracket ,
\label{tot-sym}
\end{equation}
for all $\sigma \in \mathcal{S}_{3}$ and all $\alpha _{i} \in \mathbb{F}$,  $v _{i} \in \mathbb{V}$.

Our purpose is to check if the commutator of right multiplications defines a
derivation of the ternary algebra $\mathcal{V}_{f,g,h},$ that is, considering the linear operators $R_{x}$ and $D_{x,y}$ such that
\begin{equation}
zR_{x}=zR_{(x_{1},x_{2})}=\llbracket z,x_{1},x_{2}\rrbracket ,
\end{equation}
and
\begin{equation}
D_{x,y}=\left[ R_{x},R_{y}\right] =R_{x}R_{y}-R_{y}R_{x},
\end{equation}
we want to know if the ternary version of the $ D_{x,y}$-identity,
\begin{equation}
D_{x,y}\ \llbracket z_{1},z_{2},z_{3}\rrbracket  =\llbracket
D_{x,y}\left( z_{1}\right) ,z_{2},z_{3}\rrbracket +\llbracket z_{1},D_{x,y}\left(
z_{2}\right) ,z_{3}\rrbracket +\llbracket z_{1},z_{2},D_{x,y}\left( z_{3}\right)
\rrbracket
\label{ternary derivation}
\end{equation}
holds.

Before answering this question, let us observe some immediate properties of (\ref{ternary derivation}). Indeed, it is straightforward that each linear operator $D_{(x_{1},x_{2}),(y_{1},y_{2})}$ is also linear in each $x_{i}$ and in each $y_{i}$. Further, we have the following symmetry properties:
\begin{equation*}
D_{(x_{1},x_{2}),(y_{1},y_{2})}=D_{(x_{1},x_{2}),(y_{2},y_{1})}=D_{(x_{2},x_{1}),(y_{1},y_{2})}=D_{(x_{2},x_{1}),(y_{2},y_{1})}.
\end{equation*}
Finally, it is also obvious that:
\begin{equation*}
D_{x,y}=-D_{y,x} \text{ and } D_{x,x}=0.
\end{equation*}

The following result solves the above problem.

\begin{Th}
The ternary algebra $\mathcal{V}_{f,g,h}$ is a ternary Jordan algebra if $f,\ g$ and $h$ are identically zero. In the opposite case, $\mathcal{V}_{f,g,h}$ is not a ternary Jordan algebra with the following exceptions:
\begin{itemize}
\item $\mathcal{V}_{0,0,h}$, if $char \left(\mathbb{F}\right) = 3$ and $\dim \mathbb{V} =1$;\\
\item $\mathcal{V}_{0,g,0}$, if $char \left(\mathbb{F}\right) = 2$ and $\dim \mathbb{V} =1$;\\
\item $\mathcal{V}_{f,0,h}$, if $char \left(\mathbb{F}\right) = 2$;\\
\item $\mathcal{V}_{f,g,h}$, if $char \left(\mathbb{F}\right) = 2$ and $\dim \mathbb{V} =1$.\\
\end{itemize}
\end{Th}

\begin{proof}
First, we will prove that $\mathcal{V}_{0,0,0}$ is a ternary Jordan algebra.

Obviously, for
$x=(\alpha_x+v_x, \beta_x+u_x)$, $y=(\alpha_y+v_y, \beta_y+u_y)$ and
 any $\alpha \in \mathbb{F}, z \in \mathbb{V}$,  we have:

$\llbracket \llbracket \alpha+z, \alpha_x+v_x, \beta_x+u_x \rrbracket, \alpha_y+v_y, \beta_y+u_y\rrbracket=$
$$=\alpha(\alpha_x\beta_x\alpha_y\beta_y + \alpha_x\beta_x\alpha_y u_y+\alpha_x\beta_x\beta_y v_y+\alpha_y\beta_y\alpha_xu_x+\alpha_y\beta_y\beta_xv_x)+\alpha_x\beta_x\alpha_y\beta_y z.$$
It is easy to see that the operator $D_{x,y}$ is identically zero in the algebra $\mathcal{V}_{0,0,0}$ and we have a ternary Jordan algebra.

The second part of the theorem has seven cases:
\begin{equation*}
(f\neq0,g=0,h=0), \ (f=0,g\neq0,h=0), \ \ldots, \ (f\neq0,g\neq0,h\neq0).
\end{equation*}
Below we will only consider those which lead to ternary Jordan algebras in modular characteristic, since the proof for the re\-mai\-ning ones is similar. In each, case, we will consider that $\dim \mathbb{V} =n$ is arbitrary, particularizing it whenever necessary.

{\bf I.} $(f=0,\ g=0,\ h\neq 0)$. In this case, (\ref{general-mult}) reduces to
\begin{equation*}
\llbracket \alpha_{1}+v_{1},\alpha _{2}+v_{2},\alpha _{3}+v_{3}\rrbracket =\alpha _{1}\alpha_{2}\alpha _{3} + \left( \alpha _{2}\alpha _{3}+h\left( v_{2},v_{3}\right) \right)
v_{1}+\left( \alpha _{1}\alpha _{3}+h\left( v_{1},v_{3}\right) \right)
v_{2}+\left( \alpha _{1}\alpha _{2}+h\left( v_{1},v_{2}\right) \right) v_{3}.
\end{equation*}

Let $\mathcal{S}=\left\langle 1,b\right\rangle _{\mathbb{F} }$ be a subalgebra of $\mathcal{V}_{0,0,h}$.
The multiplication table for the basis elements is given by:
\begin{equation*}
\text{(i) }\llbracket 1,1,1\rrbracket=1,\quad \text{(ii) }\llbracket 1,1,b\rrbracket=b,\quad \text{(iii) }\llbracket 1,b,b\rrbracket=0\quad \text{and}\quad \text{(iv) } \llbracket b,b,b\rrbracket=3b.
\end{equation*}

Thus, with respect to that basis, we have:
\begin{equation*}
R_{(1,1)}=E,\ R_{(1,b)}=\left(
\begin{array}{cc}
0 & 1 \\
0 & 0
\end{array}
\right)
\text{ and } R_{(b,b)}=\left(
\begin{array}{cc}
0 & 0 \\
0 & 3
\end{array}
\right).
\end{equation*}
Recalling the properties of the operators $ D_{x,y}$, in order to verify the $ D_{x,y}$-identity it is sufficient to do it for 
\begin{equation*}
D=D_{(1,b), (b,b)}=R_{(1,b)}R_{(b,b)}-R_{(b,b)}R_{(1,b)}=
\left(
\begin{array}{cc}
0 & 3 \\
0 & 0
\end{array}
\right)=3e_{12}.
\end{equation*}
Now, it is clear that $D=0$ if $char \left(\mathbb{F}\right) = 3$, and $\mathcal{S}$ is a ternary Jordan algebra. Thus, the same will occur if $n=1$.
So, admit that $char \left(\mathbb{F}\right) \neq 3$. Hereinafter, $LHS_{D}$ (respectively, $RHS_{D}$) denotes the left hand (resp., right hand) side of (\ref{ternary derivation}). It is not difficult to see that, concerning (i), we have  
\begin{equation*}
LHS_{D}=D\ \llbracket 1,1,1\rrbracket =3b \text{, while } RHS_{D}=3\llbracket D( 1),1,1\rrbracket=9b.
\end{equation*}
Therefore 
 $\mathcal{S}$ is not a ternary Jordan algebra, neither $\mathcal{V}_{0,0,h}$.

Consider again $char \left(\mathbb{F}\right) = 3$, with 
$\mathcal{V}_{0,0,h}=\left\langle 1,b_{1},...,b_{n}\right\rangle _{\mathbb{F} }$, $n>1$. Then 
\begin{equation*}
\text{(i) }\llbracket 1,1,1\rrbracket=1,\quad \text{(ii) }\llbracket 1,1,b_{i}\rrbracket=b_{i},\quad \text{(iii) }\llbracket 1,b_{i},b_{i}\rrbracket=0,  \quad \text{(iv) }\llbracket b_{i},b_{i},b_{i}\rrbracket=0, \quad \text{(v) }\llbracket 1,b_{i},b_{j}\rrbracket=0, \ (i\neq j) , \ 
\end{equation*}
$$\text{(vi) }\llbracket b_{i},b_{i},b_{j}\rrbracket=b_{j} \ (i\neq j)\quad \text{and}\quad \text{(vii) }   \llbracket b_{i},b_{j},b_{k}\rrbracket=0 \text{ ($i,j,k$ pairwise different)}.$$
Then, with respect to the considered basis, we have
\begin{equation*}
R_{(1,1)}=E,\ R_{(1,b_{i})}=e_{1,i+1}, \ R_{(b_{i},b_{i})}=e_{j+1,j+1}
\text{ and } R_{(b_{i},b_{j})}=e_{i+1,j+1}+e_{j+1,i+1} \ (i\neq j).
\end{equation*}
Taking \textit{e.g.} $b_{1}$ and $b_{2}$ (it would be similar for arbitrary choices of $b_{i}$ and $b_{j}$), 
\begin{equation*}
D=D_{(1,b_{1}), (b_{1},b_{2})}=e_{13}
\end{equation*}
so we have $D(1)=b_{2}$, $D\left(b_{1}\right)=D\left(b_{2}\right)=0$.  Now, concerning the product (i), it is easy to observe that  
\begin{equation*}
LHS_{D}=D\ \llbracket 1,1,1 \rrbracket =b_{2}, \text{ while } RHS_{D}=3\llbracket D( 1),1,1 \rrbracket=0.
\end{equation*}

{\bf II.} $(f=0,\ g\neq 0,\ h=0).$ Under these conditions, (\ref{general-mult}) reduces to
\begin{equation*}
\llbracket \alpha_{1}+v_{1},\alpha _{2}+v_{2},\alpha _{3}+v_{3}\rrbracket = 
\left(\alpha _{1}\alpha
_{2}\alpha _{3}+g\left( v_{1},v_{2},v_{3}\right) \right)+  \alpha _{2}\alpha _{3} v_{1}+\alpha _{1}\alpha _{3}
v_{2}+\alpha _{1}\alpha _{2}v_{3}.
\end{equation*}
Similarly to case I, let $\mathcal{S}=\left\langle 1,b\right\rangle _{\mathbb{F} }$ be a subalgebra of $\mathcal{V}_{0,g,0}$. The multiplication table for the basis elements is given by:
\begin{equation*}
\text{(i) }\llbracket 1,1,1\rrbracket=1,\quad \text{(ii) }\llbracket 1,1,b\rrbracket=b,\quad \text{(iii) }\llbracket 1,b,b\rrbracket=0\quad \text{and}\quad \text{(iv) } \llbracket b,b,b\rrbracket=1.
\end{equation*}
Thus, with respect to that basis, we have:
\begin{equation*}
R_{(1,1)}=E,\ R_{(1,b)}=\left(
\begin{array}{cc}
0 & 1 \\
0 & 0
\end{array}
\right)
\text{ and } R_{(b,b)}=\left(
\begin{array}{cc}
0 & 0 \\
1 & 0
\end{array}
\right).
\end{equation*}
Henceforth, in order to verify the $ D_{x,y}$-identity it is sufficient to do it for 
\begin{equation*}
D=D_{(1,b),(b,b)}=R_{(1,b)}R_{(b,b)}-R_{(b,b)}R_{(1,b)}=e_{11}-e_{22}.
\end{equation*}
Thus, $D(1)=1$ and $D(b)=-b$. Checking the $ D_{x,y}$-identity in the four cases of the multiplication table, it is possible to observe that in case (iii) that identity is always satisfied, while in the remaining ones it will only hold if $char \left(\mathbb{F}\right) = 2$. So, under this condition, $\mathcal{S}$ will be a ternary Jordan algebra and the same will happen with  $\mathcal{V}_{0,g,0}$ if $\dim \mathbb{V}=1$.

Consider now $char \left(\mathbb{F}\right) = 2$ and $\dim \mathbb{V}=2$. Take $\mathcal{V}_{0,g,0}=\left\langle 1,b_{1},b_{2} \right\rangle _{\mathbb{F} }$, with the following multiplication table
\begin{equation*}
\llbracket 1,1,1\rrbracket=1,\quad \llbracket 1,1,b_{1}\rrbracket=b_{1},\quad \llbracket 1,1,b_{2}\rrbracket=b_{2} ,\quad \llbracket 1,b_{1},b_{1}\rrbracket=\llbracket 1,b_{2},b_{2}\rrbracket=0,\quad  \llbracket b_{1},b_{1},b_{1}\rrbracket=\llbracket b_{2},b_{2},b_{2}\rrbracket=1
\end{equation*}
\begin{equation*}
\text{and } \llbracket 1,b_{1},b_{2}\rrbracket =\llbracket b_{1},b_{1},b_{2}\rrbracket=\llbracket b_{2},b_{2},b_{1}\rrbracket = 0.
\end{equation*}
Taking $D=D_{(1,b_{1}), (b_{1},b_{1})}$, we obtain: $D(1)=1$, $D(b_{1})=-b_{1}$ and $D(b_{2})=0$. Therefore, $LHS_{D}=1$  while $ RHS_{D}=0$. 

Thus,  $\mathcal{V}_{0,g,0}$ is a ternary Jordan algebra only if $char \left(\mathbb{F}\right) = 2$ and $\dim \mathbb{V} =1$.

{\bf III.} $(f\neq 0,\ g = 0,\ h\neq 0)$. In this case, (\ref{general-mult}) reduces to
\begin{equation*}
\llbracket \alpha_{1}+v_{1},\alpha _{2}+v_{2},\alpha _{3}+v_{3}\rrbracket = \left( \alpha _{1}f\left( v_{2},v_{3}\right) +\alpha _{2}f\left(
v_{1},v_{3}\right) +\alpha _{3}f\left( v_{1},v_{2}\right)+\alpha _{1}\alpha
_{2}\alpha _{3} \right)
\end{equation*}
$$+\left( \alpha _{2}\alpha _{3}+h\left( v_{2},v_{3}\right) \right)
v_{1}+\left( \alpha _{1}\alpha _{3}+h\left( v_{1},v_{3}\right) \right)
v_{2}+\left( \alpha _{1}\alpha _{2}+h\left( v_{1},v_{2}\right) \right) v_{3},$$
Being $\mathcal{S}=\left\langle 1,b\right\rangle _{\mathbb{F} }$ a subalgebra of $\mathcal{V}_{f,0,h}$, the multiplication table for the basis elements is given by:
\begin{equation*}
\text{(i) }\llbracket 1,1,1\rrbracket=1,\quad \text{(ii) }\llbracket 1,1,b\rrbracket=b,\quad \text{(iii) }\llbracket 1,b,b\rrbracket=1\quad \text{and}\quad \text{(iv) } \llbracket b,b,b\rrbracket=3b.
\end{equation*}
Thus, with respect to that basis, we have:
\begin{equation*}
R_{(1,1)}=E,\ R_{(1,b)}=\left(
\begin{array}{cc}
0 & 1 \\
1 & 0
\end{array}
\right)
\text{ and } R_{(b,b)}=\left(
\begin{array}{cc}
1 & 0 \\
0 & 3
\end{array}
\right).
\end{equation*}
Now, in order to verify the $ D_{x,y}$-identity it is sufficient to do it for 
\begin{equation*}
D=D_{(1,b),(b,b)}=R_{(1,b)}R_{(b,b)}-R_{(b,b)}R_{(1,b)}=
\left(
\begin{array}{cc}
0 & -2 \\
2 & 0
\end{array}
\right)=2\left(e_{21}-e_{12}\right).
\end{equation*}
Thus, $D(1)=-2b$ and $D(b)=2$. Checking the $ D_{x,y}$-identity in the four cases of the multiplication table, it is possible to observe that in cases (iii) and (iv) that identity always holds, while for (i) and (ii) it will be verified only if $char \left(\mathbb{F}\right) = 2$. So, under this hypothesis, $\mathcal{S}$ will be a ternary Jordan algebra and the same will happen with  $\mathcal{V}_{f,0,h}$ if $\dim \mathbb{V}=1$.

Admit that $char \left(\mathbb{F}\right) = 2$ and  $\dim \mathbb{V}>1$. Take $\mathcal{V}_{f,0,h}=\left\langle 1,b_{1},...,b_{n} \right\rangle _{\mathbb{F} }$, with the following multiplication table
\begin{equation*}
\llbracket 1,1,1\rrbracket=1,\quad \llbracket 1,1,b_{i}\rrbracket=b_{i},\quad  \llbracket 1,b_{i},b_{i}\rrbracket=1,\quad  \llbracket b_{i},b_{i},b_{i}\rrbracket=b_{i} \text{ (since $char \left(\mathbb{F}\right) = 2$ )},
\end{equation*}
\begin{equation*}
\llbracket 1,b_{i},b_{j}\rrbracket = 0, \ (i\neq j), \quad \llbracket b_{i},b_{i},b_{j}\rrbracket = b_{j}, \ (i\neq j)\quad \text{and}\quad \llbracket b_{i},b_{j},b_{k}\rrbracket =0 \ (i,j,k \text{ pairwise different and }n\geq 3).
\end{equation*}
Then, with respect to the considered basis, we have
\begin{equation*}
R_{(1,1)}=R_{(b_{i},b_{i})}=E,\ R_{(1,b_{i})}=e_{1,i+1}+e_{i+1,1}
\text{ and } R_{(b_{i},b_{j})}=e_{i+1,j+1}+e_{j+1,i+1} \ (i\neq j).
\end{equation*}
In order to verify the $ D_{x,y}$-identity it is sufficient to do it for 
\begin{equation*}
D=D_{(1,b_{i}),(b_{i},b_{j})}=e_{1,j+1}-e_{j+1,1}
\end{equation*}
Then, $D(1)=b_{j}$, $D(b_{i})=0$, $D(b_{j})=-1$ ($j\neq i$) and $D(b_{k})=0$ ( for pairwise different $i,j,k$ and $n\geq 3$). Considering all possible cases of the multiplication table for elements in $\mathcal{B}$, it is not difficult to verify that the $ D_{x,y}$-identity holds. Thus, in this case $\mathcal{V}_{f,0,h}$ is a ternary Jordan algebra.

{\bf IV.} $(f\neq 0,\ g\neq 0,\ h\neq 0)$. In this case, (\ref{general-mult}) assumes its most general form. This way, considering a subalgebra $\mathcal{S}=\left\langle 1,b\right\rangle _{\mathbb{F} }$ of $\mathcal{V}_{f,g,h}$, the multiplication table for the basis elements is given by:
\begin{equation*}
\text{(i) }\llbracket 1,1,1\rrbracket=1,\quad \text{(ii) }\llbracket 1,1,b\rrbracket=b,\quad \text{(iii) }\llbracket 1,b,b\rrbracket=1\quad \text{and}\quad \text{(iv) } \llbracket b,b,b\rrbracket=1+3b.
\end{equation*}
Thus, with respect to this basis, we have:
\begin{equation*}
R_{(1,1)}=E,\ R_{(1,b)}=\left(
\begin{array}{cc}
0 & 1 \\
1 & 0
\end{array}
\right)=e_{12}+e_{21}
\text{ and } R_{(b,b)}=\left(
\begin{array}{cc}
1 & 0 \\
1 & 3
\end{array}
\right)=e_{11}+e_{21}+3e_{22}.
\end{equation*}
This way, in order to verify the $ D_{x,y}$-identity it is sufficient to do it for 
\begin{equation*}
D=D_{(1,b),(b,b)}=R_{(1,b)}R_{(b,b)}-R_{(b,b)}R_{(1,b)}=
\left(
\begin{array}{cc}
1 & 2 \\
-2 & -1
\end{array}
\right)=e_{11}-e_{22}+2\left(e_{12}-e_{21} \right).
\end{equation*}
Observe that $D(1)=1+2b$ and $D(b)=-2-b$. Concerning the above multiplication table for the basis elements, it is easy to see that  the $ D_{x,y}$-identity holds 
if $char \left(\mathbb{F}\right) = 2$. Thus $\mathcal{S}$ is not a ternary Jordan algebra unless $char \left(\mathbb{F}\right) = 2$. This justifies that  $\mathcal{V}_{f,g,h}$ is not a ternary Jordan algebra if $char \left(\mathbb{F}\right) \neq 2$. However, it will be a ternary Jordan algebra if $\dim \mathbb{V}=2$ and $char \left(\mathbb{F}\right) = 2$.

Admit now that $char \left(\mathbb{F}\right) = 2$ and $\dim \mathbb{V}=2$. Let us consider $\mathcal{V}_{f,g,h}=\left\langle 1,b_{1},b_{2}\right\rangle _{\mathbb{F} }$. Then 
\begin{equation*}
\llbracket 1,1,1\rrbracket=1,\quad \llbracket 1,1,b_{i}\rrbracket=b_{i},i=1,2,\quad  \llbracket 1,b_{i},b_{i}\rrbracket=1,i=1,2,\quad  \llbracket b_{i},b_{i},b_{i}\rrbracket=1+b_{i} ,
\end{equation*}
\begin{equation*}
\llbracket 1,b_{1},b_{2}\rrbracket = 0\quad \text{and}\quad \ \llbracket b_{i},b_{i},b_{j}\rrbracket=b_{j}, i,j=1,2,\ (i\neq j) .
\end{equation*}
Taking $D=D_{(1,b_{1}),(b_{1},b_{1})}=e_{11}-e_{22}$, we have
\begin{equation*}
D(1)=1, \ D\left(b_{1}\right)=-b_{1} \text{ and } D\left(b_{2}\right)=0.
\end{equation*}
Then $LHS_{D}=D\ \llbracket b_{2},b_{2},b_{2}\rrbracket =1$, while $RHS_{D}=0$. Thus, we will not obtain a ternary Jordan algebra if $char \left(\mathbb{F}\right) = 2$ and $\dim \mathbb{V}>1$.

The remaining $3$ cases can be proved analogously.
\end{proof}

Thus, we obtained the first examples of  ternary Jordan algebras. In the case of $\mathcal{V}_{0,0,0}$, we have a vector space $\mathbb{F} \oplus \mathbb{V}$ equipped with the following ternary multiplication:
$$\llbracket\alpha_1+v_1, \alpha_2+v_2, \alpha_3+v_3\rrbracket= \alpha_1\alpha_2\alpha_3+ \alpha_2\alpha_3v_1 +\alpha_1\alpha_3v_2+\alpha_1\alpha_2v_3,
\mbox{ where $\alpha_i \in \mathbb{F}, v_i \in \mathbb{V}$}.$$

Recall that, given a ternary algebra $\mathcal{A}$, a subalgebra of $\mathcal{A}$ is every subspace $\mathcal{S}$ of $\mathcal{A}$ such that
\begin{equation*}
\llbracket \mathcal{S},\mathcal{S},\mathcal{S}\rrbracket \subseteq \mathcal{S},
\end{equation*}
while an ideal of $\mathcal{A}$ is every subspace $\mathcal{I}$ such that
\begin{equation*}
\llbracket \mathcal{I},\mathcal{A},\mathcal{A}\rrbracket \subseteq \mathcal{I},\llbracket \mathcal{A},\mathcal{I},\mathcal{A}\rrbracket \subseteq \mathcal{I} \text{ and } \llbracket \mathcal{A},\mathcal{A},\mathcal{I}\rrbracket
\subseteq \mathcal{I}.
\end{equation*}
On the other hand, $\mathcal{A}$  is simple if it is not abelian (\textit{i.e.},  $\llbracket \mathcal{A},\mathcal{A},\mathcal{A}\rrbracket \neq \textbf{0}$) and it lacks other ideals than the trivial ones: $\textbf{0}$ and $\mathcal{A}$.

\begin{Remark} As we can see from the following part of the paper,
the ternary algebra $\mathcal{V}_{f,0,h}$ has a ternary simple Jordan subalgebra.
\end{Remark}

\begin{Lem}  
The ternary algebra $\mathcal{V}_{0,0,0}$ is not simple and
every subspace of $\mathbb{V}$ is an ideal of  $\mathcal{V}_{0,0,0}$.
Further, if  $\mathcal{I}$ is a proper ideal of $\mathcal{V}_{0,0,0}$, then $\mathcal{I}$ is a subspace of $\mathbb{V}$. 
Among the modular ternary Jordan algebras obtained in the previous theorem, only the following are simple:
\begin{itemize}
\item $\mathcal{V}_{0,g,0}$, with $char \left(\mathbb{F}\right) = 2$ and $\dim \mathbb{V} =1$;
\item $\mathcal{V}_{f,0,h}$, with $char \left(\mathbb{F}\right) = 2$ and $\dim \mathbb{V}  > 1$.
\end{itemize}
\end{Lem}

\begin{Proof}
It is easy to see that, for every subspace $\mathbb{U}$ of $\mathbb{V}$, 
$$\llbracket \mathbb{U}, \mathbb{F}\oplus \mathbb{V}, \mathbb{F}\oplus \mathbb{V} \rrbracket = \llbracket \mathbb{U}, \mathbb{F}, \mathbb{F} \rrbracket =\mathbb{U}.$$
On the other hand, 
let $\mathcal{I}$ be an ideal of $\mathcal{V}_{0,0,0}$.
If $1+v \in \mathcal{I}$, then for every $z \in \mathcal{I}$,  $ \llbracket 1+v,1,z \rrbracket=z \in \mathcal{I}$
holds and either $\mathcal{I}$ is a subspace of $\mathbb{V}$ or $\mathcal{I}=\mathbb{F}\oplus \mathbb{V}$.

Let us consider the ternary Jordan algebra  $\mathcal{V}_{0,0,h}=\left\langle 1,b\right\rangle _{\mathbb{F} }$, with $char \left(\mathbb{F}\right) = 3$. The multiplication table for the basis elements is given by:
\begin{equation*}
\text{(i) }\llbracket 1,1,1\rrbracket=1,\quad \text{(ii) }\llbracket 1,1,b\rrbracket=b,\quad \text{(iii) }\llbracket 1,b,b\rrbracket=0\quad \text{and}\quad \text{(iv) } \llbracket b,b,b\rrbracket=0.
\end{equation*}
It is clear that $\mathcal{I}=\left\langle b\right\rangle _{\mathbb{F} }$ is an ideal of $\mathcal{V}_{0,0,h}$ and so this ternary Jordan algebra is not simple.

Let us consider the ternary Jordan algebra  $\mathcal{V}_{0,g,0}=\left\langle 1,b\right\rangle _{\mathbb{F} }$, with $char \left(\mathbb{F}\right) = 2$. The multiplication table for the basis elements is given by:
\begin{equation*}
\text{(i) }\llbracket 1,1,1\rrbracket=1,\quad \text{(ii) }\llbracket 1,1,b\rrbracket=b,\quad \text{(iii) }\llbracket 1,b,b\rrbracket=0\quad \text{and}\quad \text{(iv) } \llbracket b,b,b\rrbracket=1.
\end{equation*}
Admit that $\mathcal{I}$ is an ideal of $\mathcal{V}_{0,g,0}$ and consider $x=\alpha\ 1 + \beta b \in \mathcal{I} \backslash \{0\} $. It is clear from the multiplication table that if $1\in \mathcal{I}$ or $b\in \mathcal{I}$ then $\mathcal{I}=\mathcal{V}_{0,g,0}$. This will happen if $\beta=0$ or $\alpha=0$, respectively. So, we will suppose that none of the scalars is zero.  Then 
\begin{equation*}
\llbracket x,1,b\rrbracket= \alpha b \in  \mathcal{I},
\end{equation*}
and so $ b \in  \mathcal{I}$ leading to   $\mathcal{I}=\mathcal{V}_{0,g,0}$ according to what has been written above. Thus, $\mathcal{V}_{0,g,0}$ is simple.

Consider now $\mathcal{V}_{f,0,h}=\left\langle 1,b_{1},...,b_{n} \right\rangle _{\mathbb{F} }$, with $char \left(\mathbb{F}\right) = 2$ and  $n=\dim \mathbb{V}$. The proof will be divided in two cases: $n=1$ and $n>1$.  Recall that, when $n=1$, the multiplication table with respect to the basis $\{1,b\}$ is given by
\begin{equation*}
\llbracket 1,1,1\rrbracket=1,\quad \llbracket 1,1,b\rrbracket=b,\quad  \llbracket 1,b,b\rrbracket=1\quad \text{and} \quad \llbracket b,b,b\rrbracket = b.
\end{equation*}
It is an easy task to observe that $\mathcal{I}=\left\langle 1+b\right\rangle _{\mathbb{F} }$ is an ideal of $\mathcal{V}_{f,0,h}$, whence this ternary Jordan algebra is not simple.  Admit now that $n>1$. For the sake of simplicity, we will prove this case considering $n=2$, for it can be generalized for an arbitrary value of $n\geq 2$. The multiplication table for the basis elements of $\mathcal{V}_{f,0,h}=\left\langle 1,b_{1},b_{2}\right\rangle _{\mathbb{F} }$ is given by:
$$ \llbracket 1,1,1\rrbracket=1,\quad \llbracket 1,1,b_{i}\rrbracket=b_{i},\ i=1,2,\quad  \llbracket 1,b_{i},b_{i}\rrbracket=1,\ i=1,2,$$
$$\llbracket b_{i},b_{i},b_{i}\rrbracket=b_{i},\ i=1,2, \quad \llbracket 1,b_{1},b_{2}\rrbracket=0\quad \text{and}\quad  \llbracket b_{i},b_{i},b_{j}\rrbracket=b_{j},\ i,j=1,2.
$$
Let $\mathcal{I}$ be an ideal of $\mathcal{V}_{f,0,h}$. It is clear that if any of the basis elements is in $\mathcal{I}$, then $\mathcal{I}=\mathcal{V}_{f,0,h}$ and this ternary Jordan algebra will be simple. Consider $x=\alpha\ 1+\beta_{1}\ b_{1}+\beta_{2}\ b_{2}\in \mathcal{I}\backslash\{0\}$. Then
\begin{equation*}
\llbracket x,1,1\rrbracket=x,\ y=\llbracket x,1,b_{1}\rrbracket=\alpha\ b_{1}+\beta_{1}\ 1\in \mathcal{I},\ u=\llbracket x,1,b_{2}\rrbracket=\alpha\ b_{2}+\beta_{2}\ 1\in \mathcal{I},\ v=\llbracket x,b_{1},b_{2}\rrbracket=\beta_{1}\ b_{2}+\beta_{2}\ b_{1}\in \mathcal{I}.
\end{equation*}
Admit that $\alpha =0$. Then $y=\beta_{1}\ 1\in \mathcal{I}$. If we have $\beta_{1}=0$, then $x=\beta_{2}\ b_{2}\neq 0$, and it must be $b_{2}\in \mathcal{I}$. If we have $\beta_{1}\neq 0$, then $1\in \mathcal{I}$. So, admit that $\alpha \neq 0$. Then $y=\alpha\ b_{1}+\beta_{1}\ 1\neq 0$.  If $\beta_{1}=0$, we conclude that  $b_{1}\in \mathcal{I}$. On the other hand, if $\beta_{1} \neq 0$, both scalars in $y=\alpha\ b_{1}+\beta_{1}\ 1\neq 0$ will be non-zero. Admiting that $\beta_{2}=0$, then from $v=\llbracket x,b_{1},b_{2}\rrbracket=\beta_{1}\ b_{2}\in \mathcal{I}$, we get $b_{2}\in \mathcal{I}$. If $\beta_{2}\neq 0$, then from $\llbracket y,b_{1},b_{2}\rrbracket=\alpha\ b_{2}\in \mathcal{I}\backslash\{0\}$, we get $b_{2}\in \mathcal{I}$.

Concerning $\mathcal{V}_{f,g,h}$, with $char \left(\mathbb{F}\right) = 2$ and $\dim \mathbb{V} =1$, this is perfectely similar to the subcase $\dim \mathbb{V} =1$ of the previous case.

\end{Proof}

\begin{Lem} 
Let $D$ be an artbitrary derivation of $\mathcal{V}_{0,0,0},$
then
\begin{enumerate}
    \item if  $char \left(\mathbb{F}\right) \neq 2,$ then $Der(\mathcal{V}_{0,0,0})   \cong  End(\mathbb{V})^{(-)};$
    \item if  $char \left(\mathbb{F}\right)=2$ and $dim \mathbb{V}=1,$ then $Der(\mathcal{V}_{0,0,0})   \cong  End(\mathcal{V}_{0,0,0})^{(-)};$
    \item if  $char \left(\mathbb{F}\right)=2$ and $dim \mathbb{V}>1,$ then $D(\mathbb{V})\subseteq \mathbb{V}, Der|_{\mathbb{\mathbb{V}}}(\mathcal{V}_{0,0,0}) \cong  End(\mathbb{V})^{(-)}$ and $D(1)$ may be an arbitrary element of $\mathcal{V}_{0,0,0},$ where $Der|_{\mathbb{V}}(\mathcal{V}_{0,0,0})$ is the algebra of derivations of  $\mathcal{V}_{0,0,0}$ restricted on $\mathbb{V}.$
\end{enumerate}
\end{Lem}

\begin{Proof}
Let $D$ be a derivation of the algebra $\mathcal{V}_{0,0,0}$. If $char \left(\mathbb{F}\right) \neq 2$, then it is easy to see that $D(1)=0$. Now,  given an arbitrary element $v \in \mathbb{V}$ and  writing $D(v)=v_{\mathbb{F}}+v_D, v_{\mathbb{F}}\in \mathbb{F}, v_D \in \mathbb{V}$, we have
\begin{equation*}
D(v)= D\ \llbracket v,1+v,1+v\rrbracket =D(v)+4 v_{\mathbb{F}}v
\end{equation*}
It follows that $D(\mathbb{V}) \subseteq \mathbb{V}.$
For a mapping $D\in End(\mathbb{V}),$ we consider $D$ as a linear mapping from $End(\mathcal{V}_{0,0,0}),$ such that $D(1)=0.$
Now
$$D\ \llbracket \alpha_1+v_1, \alpha_2+v_2, \alpha_3+v_3 \rrbracket =D(\alpha_1\alpha_2\alpha_3+\alpha_1\alpha_2v_3+ \alpha_1\alpha_3v_2+\alpha_2\alpha_3v_1)$$
\begin{eqnarray*}
= & \alpha_1\alpha_2D(v_3)+ \alpha_1\alpha_3D(v_2)+\alpha_2\alpha_3D(v_1) \\
= & \llbracket \alpha_1+ v_1, \alpha_2+v_2, D(v_3) \rrbracket+\llbracket \alpha_1+v_1, D(v_2), \alpha_3+v_3 \rrbracket+\llbracket D(v_1), \alpha_2+v_2, \alpha_3+v_3 \rrbracket  \\
= & \llbracket D(\alpha_1+v_1), \alpha_2+v_2, \alpha_3+v_3 \rrbracket+\llbracket \alpha_1+v_1, D(\alpha_2+v_2), \alpha_3+v_3 \rrbracket+\llbracket \alpha_1+ v_1, \alpha_2+v_2, D(\alpha_3+v_3) \rrbracket.
\end{eqnarray*}
It follows that $D$ is a derivation of $\mathcal{V}_{0,0,0}$.

Suppose that $char \left(\mathbb{F}\right) = 2$ and admit that $\mathcal{V}_{0,0,0}=\left\langle 1,b\right\rangle _{\mathbb{F} }$.  The multiplication table for the basis elements is then given by:
\begin{equation*}
\llbracket 1,1,1\rrbracket=1,\quad \llbracket 1,1,b\rrbracket=b,\quad \llbracket 1,b,b\rrbracket= \llbracket b,b,b\rrbracket=0.
\end{equation*}
Let $D\in Der\left(\mathcal{V}_{0,0,0}\right)$. It is clear that now we can have $D(1)\neq 0$. Thus, we may set $D(1)=\alpha\ 1+ \beta\ b$ and $D(b)=\alpha '\ 1+ \beta'\ b$ for some scalars $\alpha,\beta,\alpha ',\beta '$. Applying the definition of ternary derivation to the our cases of the multiplication table, it is possible to conclude that these scalars  are arbitrary. So, $D$ can be any endomorphism of $\mathcal{V}_{0,0,0}$. 

Let us consider the case $char \left(\mathbb{F}\right) = 2$ and $\dim \mathbb{V} >1$. 
Admit, for the sake of simplicity, that $\mathcal{V}_{0,0,0}=\left\langle 1,b_{1},b_{2}, \ldots, b_n \right\rangle _{\mathbb{F} }$
 The multiplication table for the basis elements is then given by:
\begin{equation*}
\llbracket 1,1,1\rrbracket=1,\quad \llbracket 1,1,b_{i}\rrbracket=b_{i},
\end{equation*}
being null all other products. Let $D\in Der\left(\mathcal{V}_{0,0,0}\right)$. We may set 
\begin{eqnarray*}
D\left(b_{i}\right)&=&\alpha_i 1+ \beta_i.
\end{eqnarray*}
From the identity
\begin{equation*}
D\left(\llbracket x,y,z\rrbracket\right)=\llbracket D(x),y,z\rrbracket+\llbracket x,D(y),z\rrbracket+\llbracket x,y,D(z)\rrbracket
\end{equation*}
applied to all possible products of the above table we get no restrictions on the scalars, with one exception: the case $\llbracket 1,b_{i},b_{j}\rrbracket=0$. Indeed, from
\begin{equation*}
D(\llbracket 1,b_{i},b_{j}\rrbracket)=\llbracket D(1),b_{i},b_{j}\rrbracket+\llbracket 1,D(b_{i}),b_{j}\rrbracket+\llbracket 1,b_{i},D(b_{j})\rrbracket
\end{equation*}
it is easy to see, that for $i\neq j,$  we obtain: $\alpha_i=0.$ Thus, $D(V)\subseteq V$ and for any element $b_i$ the image $D(b_i)$ may be an arbitrary element of $\mathbb{V}.$

In each case, the reciprocal assertion of the isomorphism is trivial.

\end{Proof}

\section{Another example of a simple ternary Jordan algebra}

\bigskip Admit that we restrict the algebra of the previous section to $\mathbb{V}$, an $n$-dimensional vector space over a field $\mathbb{F}$, with $char \left(\mathbb{F}\right) = 0$, and denote the bilinear form $h$ by $\left( .,.\right)$, with the same properties with respect to
a given basis $\mathcal{B}=\left\{ b_{1},\ldots ,b_{n}\right\} $ of $\mathbb{V}$. Consider the following ternary multiplication defined on $\mathbb{V}$:
\begin{equation}
\llbracket x,y,z \rrbracket =\left( y,z\right) x+\left( x,z\right) y+\left(
x,y\right) z.  \label{TJA-multiplication}
\end{equation}
Denote the obtained ternary algebra by $\mathbb{A}$. It is clear that (\ref{TJA-multiplication}) is a particular case of the general multiplication (\ref{general-mult}). 

Further, when $n=4$ it is interesting to observe that (\ref{TJA-multiplication}) can be seen
as a multiple of the symmetrization of the multiplication
\begin{equation*}
\left\{ x,y,z\right\} =\frac{1}{6}\left( -\left( y,z\right) x+\left(
x,z\right) y-\left( x,y\right) z+\left[ x,y,z\right] \right) ,
\end{equation*}
defined on the ternary Filippov algebra $A_{1}$ with anticommutative multiplication $[.,.,.]$ (see \cite
{PDB_etal}). Indeed, being
\begin{eqnarray*}
\left\{ x,y,z\right\} ^{(+)} &=&sym\left( \left\{ x,y,z\right\} \right) \\
&=&\left\{ x,y,z\right\} +\left\{ x,z,y\right\} +\left\{ y,x,z\right\}
+\left\{ y,z,x\right\} +\left\{ z,x,y\right\} +\left\{ z,y,x\right\} ,
\end{eqnarray*}
it is easy to see that
\begin{equation*}
\llbracket x,y,z \rrbracket =-3\left\{ x,y,z\right\} ^{(+)}.
\end{equation*}

Clearly, (\ref{TJA-multiplication}) defines a totally commutative multiplication on $\mathbb{A}$.
Further, adopting the notations $R_{x}$ and $D_{x,y}$ introduced in the previous section, now concerning the multiplication  (\ref{TJA-multiplication}) in $\mathbb{A}$, we have the following result.


\begin{Th}
$\mathbb{A}$ is a ternary Jordan algebra. \label{THM 1}
\end{Th}

\begin{proof}
According to the definition of ternary Jordan algebra, we must
prove that
\begin{equation}
D_{x,y}\ \llbracket z_{1},z_{2},z_{3}\rrbracket  =\llbracket
D_{x,y}\left( z_{1}\right) ,z_{2},z_{3}\rrbracket +\llbracket
z_{1},D_{x,y}\left( z_{2}\right) ,z_{3}\rrbracket +\llbracket
z_{1},z_{2},D_{x,y}\left( z_{3}\right) \rrbracket \label{der-TJA-mult}
\end{equation}
holds. Due to the linearity of $D_{x,y}$ (where $x=(x_{1},x_{2})$ and $y=(y_{1},y_{2})$, with $x_{i},y_{i}\in \mathbb{V}$) and recalling its symmetry properties stated in the previous section,
it is sufficient to verify (\ref{der-TJA-mult}) for
$z_{1},z_{2},z_{3}\in \mathcal{B}$ and in
the following cases:
\begin{equation*}
\begin{tabular}{ll}
1. & $x_{1},x_{2},y_{1},y_{2}\in \left\{
b_{i},b_{j},b_{k},b_{l}\right\} $
and are all pairwise different; \\
2. & $x_{1},x_{2},y_{1},y_{2}\in \left\{ b_{i},b_{j},b_{k}\right\} $
and only two among  these  are equal; \\
3. & $x_{1},x_{2},y_{1},y_{2}\in \left\{ b_{i},b_{j}\right\} $ and
aren't
all equal; \\
4. & $x_{1},x_{2},y_{1},y_{2}\in \left\{ b_{i}\right\} $.
\end{tabular}
\end{equation*}
Using the definition of $D_{x,y}$ and (\ref{TJA-multiplication}) it is immediate to observe that, in cases 1. and 4., (\ref{der-TJA-mult}) holds trivially, since then we have
\begin{equation*}
D_{x,y} =  0.
\end{equation*}

Considering the case 2., we have to check two subcases.

2.1. $x_{1}=x_{2}=b_{i}$, $y_{1}=b_{j}$ and $y_{2}=b_{k}.$

Under these circumstances,
\begin{eqnarray*}
D_{x,y}\left( z\right)  &=&\llbracket \llbracket z,b_{i},b_{i}\rrbracket
,b_{j},b_{k}\rrbracket -\llbracket \llbracket z,b_{j},b_{k}\rrbracket
,b_{i},b_{i}\rrbracket
\\
&=&2\left( z,b_{i}\right) \llbracket b_{i},b_{j},b_{k}\rrbracket +\llbracket
z,b_{j},b_{k}\rrbracket -\left( z,b_{j}\right) \llbracket
b_{k},b_{i},b_{i}\rrbracket
-\left( z,b_{k}\right) \llbracket b_{j},b_{i},b_{i}\rrbracket  \\
&=&\left( z,b_{j}\right) b_{k}+\left( z,b_{k}\right) b_{j}-\left(
z,b_{j}\right) b_{k}-\left( z,b_{k}\right) b_{j}=0
\end{eqnarray*}
for all $z\in \mathbb{V}$, and thus (\ref{der-TJA-mult}) holds trivially.

2.2. $x_{1}=y_{1}=b_{i}$, $x_{2}=b_{j}$ and $y_{2}=b_{k}.$

Developing $D_{x,y}\left( z\right) $, we have:
\begin{equation*}
D_{x,y}\left( z\right) =\left( z,b_{j}\right) b_{k}-\left(
z,b_{k}\right) b_{j}.
\end{equation*}
Concerning (\ref{der-TJA-mult}), it is not difficult to conclude that:
\begin{eqnarray*}
LHS_{D} &=&\left[ \left( z_{1},z_{2}\right) \left( z_{3},b_{j}\right)
+\left( z_{1},z_{3}\right) \left( z_{2},b_{j}\right) +\left(
z_{2},z_{3}\right)
\left( z_{1},b_{j}\right) \right] b_{k} \\
&&-\left[ \left( z_{1},z_{2}\right) \left( z_{3},b_{k}\right)
+\left( z_{1},z_{3}\right) \left( z_{2},b_{k}\right) +\left(
z_{2},z_{3}\right) \left( z_{1},b_{k}\right) \right] b_{j}=RHS_{D}
\end{eqnarray*}
Thus (\ref{der-TJA-mult}) holds.

Let us now analyze the third case, which will be divided in three
subcases:

\begin{equation*}
\begin{tabular}{ll}
3.1. & $x_{1}=x_{2}=b_{i}$ and $y_{1}=y_{2}=b_{j}.$ \\
3.2. & $x_{1}=y_{1}=b_{i}$ and $x_{2}=y_{2}=b_{j}.$ \\
3.3. & $x_{1}=x_{2}=y_{1}=b_{i}$ and $y_{2}=b_{j}.$
\end{tabular}
\end{equation*}

Since in the first two subcases (\ref{der-TJA-mult}) trivially holds
(for the development of $D_{x,y}$ is, in each case, identically
zero), we now check what happens in the last one. We have:
\begin{equation*}
D_{x,y}\left( z\right) =2\left[ \left( z,b_{i}\right) b_{j}-\left(
z,b_{j}\right) b_{i}\right] .
\end{equation*}
Developing both sides of (\ref{der-TJA-mult}), once again we have:
\begin{eqnarray*}
LHS_{D} &=&2\left( \left( z_{1},z_{2}\right) \left( z_{3},b_{i}\right)
+\left( z_{1},z_{3}\right) \left( z_{2},b_{i}\right) +\left(
z_{2},z_{3}\right)
\left( z_{1},b_{i}\right) \right) b_{j} \\
&&-2\left( \left( z_{1},z_{2}\right) \left( z_{3},b_{j}\right)
+\left( z_{1},z_{3}\right) \left( z_{2},b_{j}\right) +\left(
z_{2},z_{3}\right)
\left( z_{1},b_{j}\right) \right) b_{i} \\
&=&RHS_{D},
\end{eqnarray*}
which ends the proof.
\end{proof}


\begin{Th}
The ternary Jordan algebra $\mathbb{A}$ is simple, except if $\dim\ \mathbb{V}=2$ and $char \left(\mathbb{F}\right) = 2$.
\end{Th}
\begin{proof}
Let $\mathcal{B}=\left\{ b_{1},\ldots ,b_{n}\right\} $ be an orthonormal basis of $\mathbb{V}$. The assertion is trivial if $n=1$, so admit that $n\geq 2$. The multiplication table for the basis elements is given by 
\begin{equation*}
\text{(i) }\llbracket b_{i},b_{i},b_{i}\rrbracket=3b_{i}, \quad \text{(ii) }\llbracket b_{i},b_{i},b_{j}\rrbracket=b_{j}, \ (i\neq j) \text{ and}\quad \text{(iii) }\llbracket b_{i},b_{j},b_{k}\rrbracket=0 \ (i,j,k) \text{ pairwise different.}
\end{equation*}
Let $\mathbb{I}\neq \{0\}$ be an ideal of $\mathbb{A}$. Clearly, it follows from (ii) in the multiplication table that if $b_{i}\in\mathbb{I}$, then the same will happen for the remaining $b_{j},j\neq i$, and then $\mathbb{I}=\mathbb{A}$, so $\mathbb{A}$ will be simple. 

Let $z=\sum\limits_{{\tiny \begin{array}{c}r=1  \end{array}}}^p \alpha _{r} b_{r}$ be an element of $ \mathbb{I}\backslash \{0\}$ with minimal length $p\neq 1$ and $\alpha_{r}\neq 0$, $r=1,...,p$. Note that there is no loss of generality in assuming this, since it is always possible reordering the basis elements. Further, as written in the previous paragraph, the assertion would trivially follow if $p=1$. Now, we have:
\begin{equation*}
w=\llbracket z,b_{1},b_{1}\rrbracket=3\alpha_{1}b_{1}+\alpha_{2}b_{2}+...+\alpha_{p}b_{p}\in \mathbb{I}\backslash \{0\}.
\end{equation*}
Then $w-z=2\alpha_{1}b_{1}\in \mathbb{I}\backslash \{0\}$ if $char \left(\mathbb{F}\right) \neq 2$, implying that $b_{1}\in \mathbb{I}$, and thus $\mathbb{I}=\mathbb{A}$. Before considering the case $char \left(\mathbb{F}\right) = 2$, observe that, when $char \left(\mathbb{F}\right) = 3$, despite being $\llbracket b_{i},b_{i},b_{i}\rrbracket=0$ and thus 
\begin{equation*}
w=\llbracket z,b_{1},b_{1}\rrbracket=\alpha_{2}b_{2}+...+\alpha_{p}b_{p}\in \mathbb{I}\backslash \{0\},
\end{equation*}
we would arrive to the same conclusion by considering $z-w=\alpha_{1}b_{1}\in \mathbb{I}\backslash \{0\}$.

Assume that $char \left(\mathbb{F}\right) = 2$. Thus, the only difference in the multiplication table is that $\llbracket b_{i},b_{i},b_{i}\rrbracket=b_{i}$. Admit first that $\dim \mathbb{V}=2$ and $\mathcal{B}=\left\{ b_{1},b_{2}\right\}$. Let $z=b_{1}+b_{2}$. From 
\begin{equation*}
\llbracket z,b_{1},b_{1}\rrbracket=z, \quad \llbracket z,b_{2},b_{2}\rrbracket=z\quad \text{and}\quad \llbracket z,b_{1},b_{2}\rrbracket=z 
\end{equation*}
it is clear that $\mathbb{I}=\langle z\rangle_{\mathbb{F}}$ is a non-trivial ideal of $\mathbb{A}$, so $\mathbb{A}$ is not simple. Admit now that $\dim \mathbb{V}> 2$ and consider $0\neq\mathbb{I}$ an ideal of $\mathbb{A}$. As previously done, set $z=\sum\limits_{{\tiny \begin{array}{c}r=1  \end{array}}}^p \alpha _{r} b_{r} \in \mathbb{I}\backslash \{0\}$ with minimal length $p\neq 1$ and $\alpha_{r}\neq 0$, $r=1,...,p$. Then, for $i,j\in \{1,...,p\}$, $i\neq j$, we have:
\begin{equation*}
w=\llbracket z,b_{i},b_{j}\rrbracket=\alpha_{i}b_{j}+\alpha_{j}b_{i}\in \mathbb{I}\backslash \{0\}.
\end{equation*}
On the other hand, for $k\notin \{i,j\}$ we have
\begin{equation*}
w'=\llbracket w,b_{i},b_{k}\rrbracket=\alpha_{j}b_{k}\in \mathbb{I}\backslash \{0\}.
\end{equation*}
So, $b_{k}\in \mathbb{I}$ and $\mathbb{I}=\mathbb{A}$.
\end{proof}

\bigskip

Recall now that an identity satisfied by a ternary algebra is said to be of degree (or level) $k$, with $k\in \mathbb{N}$, if $k$ is the number of times that the multiplication appears in each term of the identity (see \cite{PDB_etal}).  Next, we are going to study the identities of degrees $1$ and $2$, respectively, valid in the ternary Jordan algebra  $\mathbb{A}$.


The identities of degree $1$ satisfied by $\mathbb{A}$ have the following shape:
\begin{equation*}
\sum\limits_{{\tiny \begin{array}{c}\sigma \in S_{3}  \end{array}}}\alpha _{\sigma }\llbracket x_{\sigma \left(
1\right) },x_{\sigma \left( 2\right) },x_{\sigma \left( 3\right) }\rrbracket
=0,\qquad \alpha _{\sigma }\in \mathbb{F} .
\end{equation*}
Due to the total commutativity of the multiplication  (\ref {TJA-multiplication}),  this sum reduces to one summand and it is not  difficult to observe that  the identities of degree $1$ are resumed by that property.

Again by the total commutativity of the multiplication, the degree $2$ identities  valid in $\mathbb{A}$ assume the following form:
\begin{equation*}
\sum\limits_{{\tiny \begin{array}{c}\sigma \in S_{5} \\ 
\sigma(1)<\sigma(2)<\sigma(3) \\ \sigma(4)<\sigma(5) \end{array}}}\alpha_{\sigma }\llbracket \llbracket x_{\sigma
\left( 1\right) },x_{\sigma \left( 2\right) },x_{\sigma \left( 3\right)
}\rrbracket ,x_{\sigma \left( 4\right) },x_{\sigma \left( 5\right) }\rrbracket
=0,\qquad \alpha _{\sigma }\in \mathbb{F} ,
\end{equation*}
which can be expanded in the following way:
\begin{equation}
\begin{tabular}{l}
$\alpha _{1}\llbracket \llbracket x,y,z\rrbracket ,u,v\rrbracket+
 \alpha _{2}\llbracket \llbracket x,y,u\rrbracket ,z,v\rrbracket +
 \alpha _{3}\llbracket \llbracket x,y,v\rrbracket ,z,u\rrbracket+
 \alpha _{4}\llbracket \llbracket x,z,u\rrbracket ,y,v\rrbracket +\alpha _{5}\llbracket \llbracket x,z,v\rrbracket ,y,u\rrbracket+ $ \\
$ \alpha _{6}\llbracket\llbracket x,u,v\rrbracket ,y,z\rrbracket +
  \alpha _{7}\llbracket \llbracket y,z,u\rrbracket,x,v\rrbracket +
  \alpha _{8}\llbracket \llbracket y,z,v\rrbracket ,x,u\rrbracket  
+\alpha _{9}\llbracket \llbracket y,u,v\rrbracket ,x,z\rrbracket +
  \alpha _{10}\llbracket\llbracket z,u,v\rrbracket ,x,y\rrbracket =0.$%
\end{tabular}
\label{id-deg5}
\end{equation}
Let us find what conditions must the $\alpha_{i}$ satisfy.

{\bf 1.} If $\dim \mathbb{V}=1,$ being $\mathbb{V}=\left\langle b\right\rangle _{\mathbb{F}}$, since $\llbracket
b,b,b\rrbracket =3b$ from (\ref{id-deg5}) we get:
\begin{equation}
\sum_{i=1}^{10}\alpha _{i}=0.  \label{first}
\end{equation}
Now, admit that $\mathbb{V}=\left\langle b_{1},b_{2}\right\rangle _{\mathbb{F}}$, where $\left\{ b_{1},b_{2}\right\} $ is an orthonormal basis of $\mathbb{V}$. Then
\begin{equation*}
\llbracket b_{i},b_{i},b_{i}\rrbracket =3b_{i},\text{ }i=1,2,\quad \text{ and}\quad
\text{ }\llbracket b_{i},b_{i},b_{j}\rrbracket =b_{j}\text{, }i\neq j.
\end{equation*}
In order to analyze what relations between the scalars can be derived from (%
\ref{id-deg5}) , we are going to check all non redundant possible cases with
$x,y,z,u,v$ in the considered basis.


{\bf 2.} Suppose that among $x,y,z,u,v$ only four are equal (\textit{e.g.}, to $%
b_{1}$). Then, we have do consider $5$ subcases:
\begin{equation*}
\begin{tabular}{ll}
(2.1) $x=y=z=u=b_{1}$ and $v=b_{2};$ & (2.2) $x=y=z=v=b_{1}$ and $u=b_{2};$ \\
(2.3) $x=y=u=v=b_{1}$ and $z=b_{2};$ & (2.4) $x=z=u=v=b_{1}$ and $y=b_{2};$ \\
(2.5) $y=z=u=v=b_{1}$ and $x=b_{2}.$%
\end{tabular}
\end{equation*}

Replacing in (\ref{id-deg5}) for each subcase, we obtain:
\begin{equation*}
\begin{tabular}{l}
$(2.1) \to \ 3\alpha _{1}+3\alpha _{2}+\alpha _{3}+3\alpha _{4}+\alpha _{5}+\alpha
_{6}+3\alpha _{7}+\alpha _{8}+\alpha _{9}+\alpha _{10}=0,$\\
$(2.2) \to \ 3\alpha _{1}+\alpha _{2}+3\alpha _{3}+\alpha _{4}+3\alpha _{5}+\alpha
_{6}+\alpha _{7}+3\alpha _{8}+\alpha _{9}+\alpha _{10}=0,$ \\
$(2.3) \to \ \alpha _{1}+3\alpha _{2}+3\alpha _{3}+\alpha _{4}+\alpha _{5}+3\alpha
_{6}+\alpha _{7}+\alpha _{8}+3\alpha _{9}+\alpha _{10}=0,$\\
$(2.4) \to \ \alpha _{1}+\alpha _{2}+\alpha _{3}+3\alpha _{4}+3\alpha _{5}+3\alpha
_{6}+\alpha _{7}+\alpha _{8}+\alpha _{9}+3\alpha _{10}=0,$\\
$(2.5) \to \ \alpha _{1}+\alpha _{2}+\alpha _{3}+\alpha _{4}+\alpha _{5}+\alpha
_{6}+3\alpha _{7}+3\alpha _{8}+3\alpha _{9}+3\alpha _{10}=0.$
\end{tabular}
\end{equation*}

{\bf 3.} Admit now that among $x,y,z,u,v$ only three are equal (\textit{e.g.}, to $%
b_{1}$). Then, we have to consider ten subcases:
\begin{equation*}
\begin{tabular}{ll}
(3.1) $x=y=z=b_{1}$ and $u=v=b_{2};$ & (3.2) $x=y=u=b_{1}$ and $z=v=b_{2};$ \\
(3.3) $x=y=v=b_{1}$ and $z=u=b_{2};$ & (3.4) $x=z=u=b_{1}$ and $y=v=b_{2};$ \\
(3.5) $x=z=v=b_{1}$ and $y=u=b_{2};$ & (3.6) $x=u=v=b_{1}$ and $y=z=b_{2};$ \\
(3.7) $y=z=u=b_{1}$ and $x=v=b_{2};$ & (3.8) $y=z=v=b_{1}$ and $x=u=b_{2};$ \\
(3.9) $y=u=v=b_{1}$ and $x=z=b_{2};$ & (3.10) $z=u=v=b_{1}$ and $x=y=b_{2}.$%
\end{tabular}
\end{equation*}
Analogously to what we have done in case 2., we will obtain the following
equations:
\begin{equation*}
\begin{tabular}{l}
$(3.1) \to \ 3\alpha _{1}+\alpha _{2}+\alpha _{3}+\alpha _{4}+\alpha _{5}+3\alpha
_{6}+\alpha _{7}+\alpha _{8}+3\alpha _{9}+3\alpha _{10}=0,$\\
$(3.2) \to \ \alpha _{1}+3\alpha _{2}+\alpha _{3}+\alpha _{4}+3\alpha _{5}+\alpha
_{6}+\alpha _{7}+3\alpha _{8}+\alpha _{9}+3\alpha _{10}=0,$\\
$(3.3) \to \ \alpha _{1}+\alpha _{2}+3\alpha _{3}+3\alpha _{4}+\alpha _{5}+\alpha
_{6}+3\alpha _{7}+\alpha _{8}+\alpha _{9}+3\alpha _{10}=0,$\\
$(3.4) \to \ \alpha _{1}+\alpha _{2}+3\alpha _{3}+3\alpha _{4}+\alpha _{5}+\alpha
_{6}+\alpha _{7}+3\alpha _{8}+3\alpha _{9}+\alpha _{10}=0,$\\
$(3.5) \to \ \alpha _{1}+3\alpha _{2}+\alpha _{3}+\alpha _{4}+3\alpha _{5}+\alpha
_{6}+3\alpha _{7}+\alpha _{8}+3\alpha _{9}+\alpha _{10}=0,$\\
$(3.6) \to \ 3\alpha _{1}+\alpha _{2}+\alpha _{3}+\alpha _{4}+\alpha _{5}+3\alpha
_{6}+3\alpha _{7}+3\alpha _{8}+\alpha _{9}+\alpha _{10}=0,$\\
$(3.7) \to \ \alpha _{1}+\alpha _{2}+3\alpha _{3}+\alpha _{4}+3\alpha _{5}+3\alpha
_{6}+3\alpha _{7}+\alpha _{8}+\alpha _{9}+\alpha _{10}=0,$\\
$(3.8) \to \ \alpha _{1}+3\alpha _{2}+\alpha _{3}+3\alpha _{4}+\alpha _{5}+3\alpha
_{6}+\alpha _{7}+3\alpha _{8}+\alpha _{9}+\alpha _{10}=0,$\\
$(3.9) \to \ 3\alpha _{1}+\alpha _{2}+\alpha _{3}+3\alpha _{4}+3\alpha _{5}+\alpha
_{6}+\alpha _{7}+\alpha _{8}+3\alpha _{9}+\alpha _{10}=0,$\\
$(3.10) \to \ 3\alpha _{1}+3\alpha _{2}+3\alpha _{3}+\alpha _{4}+\alpha _{5}+\alpha
_{6}+\alpha _{7}+\alpha _{8}+\alpha _{9}+3\alpha _{10}=0.$
\end{tabular}
\end{equation*}

Since $\dim \mathbb{V}=2$, it is clear that all other cases are redundant. Now, the
linear system consisting of (\ref{first}) and the other $15$ equations has
only the trivial solution. Therefore, the only identities of degree $2$ in $\mathbb{A}
$ are those that result from lifting the identities of degree $1$.

\begin{Remark}
Observe that a lifting is every process which allows to obtain $(k+1)$-degree identities starting form $k$-degree identities. This include techniques of two types (written in terms of the ternary multiplication $\llbracket .,.,. \rrbracket$):\newline
(i) embedding -- which justifies, \textit{e.g.}, that
\begin{equation*}
\llbracket \llbracket a,b,c \rrbracket ,d,e\rrbracket =\llbracket \llbracket b,a,c\rrbracket ,d,e\rrbracket \text{ starting from } \llbracket a,b,c\rrbracket
=\llbracket b,a,c\rrbracket ,
\end{equation*}
\newline
(ii)  replacing an element by a triple -- justifying, \textit{%
e.g.}, that
\begin{equation*}
\llbracket \llbracket a,b,c\rrbracket,d,e\rrbracket =\llbracket \llbracket a,b,c\rrbracket ,e,d\rrbracket \text{ starting from } \llbracket a,b,c\rrbracket =\llbracket a,c,b\rrbracket .
\end{equation*}
\end{Remark}

Thus, we have the following results:

\begin{Lem}
All degree $1$ identities on $\mathbb{A}$ are a consequence of the total commutativity of  (\ref {TJA-multiplication}).
\end{Lem}

\begin{Lem}
All degree $2$ identities on $\mathbb{A}$ are a consequence of the total commutativity of  (\ref {TJA-multiplication}),
by means of a lifting process.
\end{Lem}

\begin{Remark}
Recall that  a Jordan triple system (see \cite{bremner_peresi} and also \cite
{gnedbaye2007}, where this notion also appears under the name of ''ternary Jordan algebra'') is a ternary algebra $\mathbb{A}$ with ternary multiplication $\llbracket .,.,. \rrbracket$ satisfying a partial commutativity property \label{other}
\begin{equation*}
\llbracket x,y,z \rrbracket = \llbracket z,y,x \rrbracket
\end{equation*}
and the following identity:
\begin{equation}\label{tjs}
\llbracket \llbracket x,y,z \rrbracket ,u,v \rrbracket +\llbracket z,u,\llbracket x,y,v \rrbracket\rrbracket =\llbracket x,y,\llbracket z,u,v \rrbracket\rrbracket +\llbracket z,\llbracket y,x,u \rrbracket ,v\rrbracket.
\end{equation}
According to the previous computations, it is also clear that this ternary Jordan
algebra $\mathbb{A}$ doesn't satisfy $(\ref{tjs})$, clarifying that we are working with a different generalization.
\end{Remark}


\section{Derivations of the ternary Jordan algebra $\mathbb{A}$}

We are now going to describe the derivations of $\mathbb{A}$, the ternary Jordan
algebra defined in the previous section.
Consider a linear map $D:\mathbb{V}\to \mathbb{V}$. Using the definition of
derivation of a ternary algebra and (\ref{TJA-multiplication}), it is possible to see
that $D\in Der\left( \mathbb{A}\right) $ if and only if
\begin{equation}
\big( \left( D\left( y\right) ,z\right) +\left( y,D(z)\right) \big) x+%
\big( \left( D\left( x\right) ,z\right) +\left( x,D(z)\right) \big) y+%
\big( \left( D\left( x\right) ,y\right) +\left( x,D(y)\right) \big) z=0,
\label{derivation-TJA}
\end{equation}
for all $x,y,z\in \mathbb{V}$. It is clear that it is sufficient to work with (\ref
{derivation-TJA}) for all $x,y,z\in \mathcal{B}$, the orthonormal basis of $\mathbb{V}$ we have chosen before. 



It is easy to see that
\begin{equation}
\left( D\left( b_{i}\right) ,b_{j}\right) =-\left( D\left( b_{j}\right),b_{i}\right), \text{ with } i\neq j.  \label{second_der-TJA}
\end{equation}
This way, we have:
\begin{equation*}
D\left( b_{j}\right) =\sum_{i=1}^{n}\alpha _{i}b_{i}=\sum_{i=1}^{n}\left(b_{i},D\left( b_{j}\right) \right) b_{i}.
\end{equation*}
This means that for every derivation $D$, it is true that $D$ is a skewsymmetric operator of $\mathbb{V}$.

Now, observe that $Inder\left( \mathbb{A}\right) $, the algebra of inner derivations
of $\mathbb{A}$, is just the Lie algebra generated by the the right multiplication
operators $R_{x}$, $x=(x_{1},x_{2})$ , or, equivalently,
\begin{equation*}
Inder\left( \mathbb{A}\right) =\left\langle D_{x,y}:x=\left( x_{1},x_{2}\right)
,y=\left( y_{1},y_{2}\right) ,x_{i},y_{i}\in \mathbb{V}\right\rangle _{\mathbb{F} },
\end{equation*}
where $D_{x,y}$ is defined as before. Thus, we are going to
consider the operators $D_{x,y}$ and analyze the only cases when these are
non trivially null. Recalling the proof of Theorem \ref{THM 1}, this means
that we just have to see what happens in the subcases 2.2 and  3.3.

As in subcase 2.2, let now $x_{1}=y_{1}=b_{i}$, $x_{2}=b_{j},$ $y_{2}=b_{k}$  and consider $i,j,k$ pairwise different.
Recall that we have:
\begin{equation*}
D_{x,y}(z)=\left( z,b_{j}\right) b_{k}-\left( z,b_{k}\right) b_{j},
\end{equation*}
and thus
\begin{equation*}
D_{x,y}(b_{j})=b_{k},\quad D_{x,y}(b_{k})=-b_{j}\quad \text{and}\quad
D_{x,y}(b_{r})=0,\text{ }r\neq k,j.
\end{equation*}
This way, if $M_{x,y}$\ denotes the matrix of $D_{x,y}$ with respect to $%
\mathcal{B}$, it is clear that
\begin{equation*}
M_{x,y}=M_{\left( b_{i},b_{j}\right) ,\left( b_{i},b_{k}\right)
}=e_{jk}-e_{kj},\quad j\neq k.
\end{equation*}

Now, analogously to the subcase 3.3., let us take  $x_{1}=x_{2}=y_{1}=e_{i}$ and $y_{2}=e_{j}$, $i\neq j$. Then, we know that
\begin{equation*}
D_{x,y}(z)=2\big( \left( z,b_{i}\right) b_{j}-\left( z,b_{j}\right) b_{i}
\big) ,
\end{equation*}
and thus
\begin{equation*}
D_{x,y}(b_{i})=2b_{j},\quad D_{x,y}(b_{j})=-2b_{i}\quad \text{and}\quad
D_{x,y}(b_{k})=0,\text{ }k\neq i,j.
\end{equation*}
This means that
\begin{equation*}
M_{x,y}=M_{\left( b_{i},b_{i}\right) ,\left( b_{i},b_{j}\right) }=2\left(
e_{ij}-e_{ji}\right) ,\quad i\neq j.
\end{equation*}

From here and from the above caracterization of the derivations of $\mathbb{A}$, it
is clear that the following result holds:

\begin{Th}  \label{der}
$Der\left( \mathbb{A}\right) =Inder\left( \mathbb{A}\right)=so(n).$
\end{Th}

\begin{Remark}
In $1955$ Jacobson proved that if a finite dimensional Lie algebra over a field of characteristic zero has an invertible derivation,
then it is a nilpotent algebra \cite{jac}.
The same result was proved for Jordan algebras \cite{kp16}, but as we can see from Theorem \ref{der}
the Theorem of Jacobson is not true for ternary Jordan algebras.
We can take a ternary Jordan algebra $\mathbb{A}$ (as in Theorem \ref{der}) with dimension $4$ and
consider the map defined by the following matrix $\sum_{1\leq i<j\le 4} (e_{ij}-e_{ji}).$
As follows, there is a simple  ternary Jordan algebra with an invertible derivation.
\end{Remark}

\section{Searching for new examples}

In this section, we give three examples of ternary algebras that appeared while searching for new intersting examples of ternary Jordan algebras.

\subsection{Ternary symmetrized matrix algebras}

Consider the following ternary algebras
\begin{equation*}
\mathfrak{A}=\left( M_{n}\left(\mathbb{F}\right),\llbracket .,.,.\rrbracket \right),
\end{equation*}
where $\llbracket .,.,.\rrbracket$ is the symmetrized ternary multiplication defined by
\begin{equation*}
\llbracket A,B,C\rrbracket=sym\left(ABC\right)=ABC+ACB+BAC+BCA+CAB+CBA\text{, with }A,B,C\in
M_{n}\left(\mathbb{F}\right) .
\end{equation*}

This multiplication, also known as the ternary anticommutator, is clearly total commutative.  It is also a simple task to verify that $\mathfrak{A}$ is not a ternary Jordan algebra (at least, if $char \left(\mathbb{F}\right) \neq 3$). In fact, let $x=\left(x_{1},x_{2}\right)$, $y=\left(y_{1},y_{2}\right)$, with $x_{1},x_{2},y_{1},y_{2}\in M_{n}\left(\mathbb{F}\right)$ and consider $D_{x,y}$ as defined in the previous sections. The identity 
\begin{equation*}
D_{x,y}\ \llbracket A,B,C\rrbracket=\llbracket D_{x,y}(A),B,C\rrbracket+\llbracket A, D_{x,y}(B),C\rrbracket+\llbracket A,B,D_{x,y}(C)\rrbracket
\end{equation*}
is not satisfied in $\mathfrak{A}$. To see this, we can consider $n=3$ and evaluate both sides of the identity for the following elements of the canonical basis of $ M_{n}\left(\mathbb{F}\right)$:
\begin{equation*}
x_{1}=e_{23},\  x_{2}=e_{32},\ y_{1}=e_{22},\ y_{2}=e_{23},\ A=e_{12},\ B=e_{23}\text{ and } C=e_{32}.
\end{equation*}
Then $LHS_{D}=0$, while $RHS_{D}=-3e_{13}$.

However, we have the following result:

\begin{Th}
Given different $i,j\in \left\{1,...,n\right\}$ the following $2$-dimensional subalgebras of $M_{n}\left(\mathbb{F}\right)$ 
\begin{equation*}
\mathfrak{S}_{1}=\left\langle e_{ii},e_{ij} \right\rangle _{\mathbb{F}} 
\text{and}\quad \mathfrak{S}_{2}=\left\langle e_{ij},e_{ji} \right\rangle _{\mathbb{F}},\quad (i\neq j), 
\end{equation*}
are non-isomorphic ternary Jordan subalgebras of $\mathfrak{A}$. Further, 
$\mathfrak{S}_{2}$ is simple.
\end{Th}
\begin{Proof}
The proof of the first assertion will only be done in the case of the subalgebra $\mathfrak{S}_{1}$, since the other case could be proved analogously.

The multiplication table for the basis elements of $\mathfrak{S}_{1}$ is given by:
\begin{equation*}
\llbracket e_{ii},e_{ii},e_{ii}\rrbracket=6e_{ii},\quad \llbracket e_{ii},e_{ii},e_{ij}\rrbracket=2e_{ij}\quad \text{and}\quad \llbracket e_{ii},e_{ij},e_{ij}\rrbracket=\llbracket e_{ij},e_{ij},e_{ij}\rrbracket=0.
\end{equation*}
Thus, considering the matrix representation of the right multiplication operators $R_{\left(e_{ii},e_{ii}\right)}$,  $R_{\left(e_{ii},e_{ij}\right)}$ and $R_{\left(e_{ij},e_{ij}\right)}$ with respect to the basis $\left\{e_{ii},e_{ij}\right\}$, we have:
\begin{equation*}
R_{\left(e_{ii},e_{ii}\right)}=
\left(
\begin{array}{cc}
6 & 0 \\
0 & 2
\end{array}
\right), \quad
R_{\left(e_{ii},e_{ij}\right)}=
\left(
\begin{array}{cc}
0 & 2 \\
0 & 0
\end{array}
\right) \quad \text{and}\quad
R_{\left(e_{ij},e_{ij}\right) }=0.
\end{equation*}
Defining $D_{(x_{1},x_{2}),(y_{1},y_{2})}$ as before, the only non-trivial case for these operators is given by
\begin{equation*}
D=D_{\left(e_{ii},e_{ii}\right),\left(e_{ii},e_{ij}\right)}=
\left(
\begin{array}{cc}
0 & 8 \\
0 & 0
\end{array}
\right)
\end{equation*}
(or a scalar multiple of this), which means that
\begin{equation*}
D\left(e_{ii}\right)=8e_{ij} \text{ and } D\left(e_{ij}\right)=0.
\end{equation*}
Knowing this and due to the symmetry properties of the operator $D$, for the verification of the $D_{x,y}$-identity it is sufficient to do it in the four cases of the multiplication table above. In the first case, we will have
\begin{equation*}
LHS_{D}=D\ \llbracket e_{ii},e_{ii},e_{ii}\rrbracket =48e_{ij} =3\llbracket D\left(e_{ii}\right),e_{ii},e_{ii}\rrbracket=RHS_{D}.
\end{equation*}
For the other three cases, both sides of the identity will be null, proving that  $\mathfrak{S}_{1}$  is a ternary Jordan algebra. 

Finally, observing the multiplication table for the basis elements of each subalgebra $\mathfrak{S}_{i}, i=1,2$, it is an easy task to prove that $\left\langle e_{ij} \right\rangle _{\mathbb{F}} $ is an ideal of $\mathfrak{S}_{1}$.

Admit now that $\mathfrak{I}$ is an ideal of $\mathfrak{S}_{2}$ and consider $x=\alpha e_{ij}+ \beta e_{ji} \in \mathfrak{I}\backslash \{0\}$.  Observing that the multiplication table for the basis elements of $\mathfrak{S}_{2}$ is given by:
\begin{equation*}
\llbracket e_{ij},e_{ij},e_{ij}\rrbracket=\llbracket e_{ji},e_{ji},e_{ji}\rrbracket=0,\ \llbracket e_{ij},e_{ij},e_{ji}\rrbracket=2e_{ij} \text{ and } \llbracket e_{ij},e_{ji},e_{ji}\rrbracket=2e_{ji},
\end{equation*}
we have
\begin{equation*}
\llbracket x, e_{ij},e_{ij}\rrbracket=2\beta e_{ij}.
\end{equation*}
If $\beta \neq 0$, then  $e_{ij}\in \mathfrak{I}$ (we avoid the case $char \left(\mathbb{F}\right) = 2$ which would lead to an identically zero multiplication). This way,
\begin{equation*}
\llbracket e_{ij}, e_{ji},e_{ji}\rrbracket =2 e_{ji}\in \mathfrak{I}
\end{equation*}
which implies that $e_{ji}\in \mathfrak{I}$ and thus $\mathfrak{I}=\mathfrak{S}_{2}$. If $\beta =0$, then  $x=\alpha e_{ij}\in \mathfrak{I}\backslash \{0\}$ implies $e_{ij}\in \mathfrak{I}$ and we arrive to the same conclusion.

Finally, it is clear that these two algebras are not isomorphic.

\end{Proof}

Concerning the identities verified in $\mathfrak{A}$, it is possible to prove the same results we have achieved about the algebra in the third section, by using similar techniques. This means that:
\begin{itemize}
    \item the identities of degree $1$ satisfied by $\mathfrak{A}$ are  resumed in its total commutativity;
    \item all degree $2$ identities satisfied by  $\mathfrak{A}$ result from lifting the total commutativity of the anticommutator. 
\end{itemize}

\subsection{Ternary algebras defined on the Cayley-Dickson algebras}

The Cayley-Dickson doubling process, \cite{Schafer_article}, can give us new examples of ternary Jordan algebras. Let us recall such process. Consider a unital algebra $\mathcal{A}$ over a field $\mathbb{F}$,  $char \left(\mathbb{F}\right) =0$,  equipped with an involution $x\mapsto \overline{x}$ such that
\begin{equation*}
x+\overline{x},\ x \overline{x}\in \mathbb{F} \text{, for all }x\in \mathcal{A}.
\end{equation*}
Let $a\in \mathbb{F} \backslash \left\{ 0\right\} $ and define a new algebra $%
\left( \mathcal{A},a\right) $ as follows:
\begin{equation*}
\begin{tabular}{ll}
$\mathcal{A}\oplus \mathcal{A}$, & the underlying vector space, \\
 $\left(
x_{1},x_{2}\right) +\left( y_{1},y_{2}\right) =\left(
x_{1}+y_{1},x_{2}+y_{2}\right), $ &the addition, \\
$c\left( x_{1},x_{2}\right) =\left( cx_{1},cx_{2}\right), $ &the
scalar multiplication, \\ $\left( x_{1},x_{2}\right) \left( y_{1},y_{2}\right) =\left(
x_{1}y_{1}+ay_{2}\overline{x_{2}},\overline{x_{1}}y_{2}+y_{1}x_{2}\right),
$ &the multiplication.
\end{tabular}
\end{equation*}
The corresponding involution is given by:
\begin{equation*}
\overline{\left( x_{1},x_{2}\right) }=\left(
\overline{x_{1}},-x_{2}\right) .
\end{equation*}
Starting with $\mathbb{F} $ such that $char\left( \mathbb{F} \right) \neq 2$, we obtain a sequence of $2^{t}$-dimensional algebras denoted by $\mathcal{U}_{t}$, among which:
\begin{equation*}
\begin{tabular}{lll}
$\mathcal{U}_{0}=\mathbb{F}, $ &the scalars, &  commutative and associative, \\
 $\mathcal{U}_{1}=\mathbb{C}\left( a\right) =\left( \mathbb{F}
,a\right), $ & generalized complex numbers, &  commutative and associative, \\
$\mathcal{U}_{2}=\mathbb{H}\left( a,b\right) =\left(
\mathbb{C}\left( a\right)
,b\right), $ &the generalized quaternions, & not commutative and associative, \\
$\mathcal{U}_{3}=\mathbb{O}\left( a,b,c\right) =\left(
\mathbb{H}\left( a,b\right) ,c\right), $ &
the generalized octonions, &  not commutative, not associative and alternative, \\
\end{tabular}
\end{equation*}
are the most notable examples. Define on each $\mathcal{U}_{t}$, $t=2,3,...$ the ternary
multiplication:
\begin{equation}
\llbracket x,y,z\rrbracket =\left( x\overline{y}\right) z
\label{CD-ternarymult}
\end{equation}
and take
\begin{equation*}
\mathcal{D}_{t}=\left( \mathcal{U}_{t},\llbracket .,.,.\rrbracket \right)
.
\end{equation*}
Clearly, this ternary mutiplication is not totally commutative, so these algebras are not ternary Jordan algebras. Before going on, we will recall some properties of compositon alegbras (thus, valid in  particular in $\mathcal{U}_{2}$ and $\mathcal{U}_{3}$).

\begin{Lem}
Let $\mathbb{A}$ be a composition algebra with identity $1$, with an involution $\overline{ }$ and a bilinear symmetric non-degenerate form $\langle.,.\rangle$.  
For any alements $a,b,c\in \mathbb{A}$, we have
\begin{enumerate}
\item $a\overline{a}b=a(\overline{a}b)=n(a)b=b\overline{a}a=b(\overline{a}a)$;
\item $a\overline{b}c+a\overline{c}b=2\langle b,c\rangle a$;
\item $a(\overline{b}c)+b(\overline{a}c)=2\langle a,b\rangle c$.
\\ If, aditionally, $a,b,c$ are different elements in an orthonormal basis, then:
\item $\overline{a}b\overline{a}=-\overline{b}$;
\item $a\overline{b}c=-a\overline{c}b$;
\item $a(\overline{b}c)=-b(\overline{a}c)$.
\end{enumerate}
\end{Lem}

Let us forget total commutativity. Note that $D_{x,y}=-D_{y,x}$ trivially holds on each $\mathcal{D}_{t}$ . Under these circumstances, we have the following results.

\begin{Th} 
$\mathcal{D}_{2}$ is a simple ternary $D_{x,y}$-derivation algebra.
\end{Th}

\begin{proof}
Consider arbitrary $x=(x_{1},x_{2})$, $y=(y_{1},y_{2})$, with $x_{i},y_{i}\in \mathbb{H}( a,b )$, $i=1,2$. It is clear that every linear operator $D_{x,y}$ is also linear in each $x_{i}$ and each $y_{i}$, so we can consider that these elements belong to $\mathcal{B}=\left\{1, a,b,ab\right\}$, the usual orthonormal basis of $\mathbb{H}(a,b)$. Before going further, let us recall that $\mathcal{U}_{2}$ is associative and let us state some properties of $\llbracket .,.,.\rrbracket$ and of each operator $D_{x,y}$ in $\mathcal{D}_{2}$.

First of all, by the previous lemma, note that for pairwise different elements $x,y,z\in \mathcal{B}$, we have:
\begin{equation*}
\llbracket x,y,z\rrbracket =-\llbracket y,x,z\rrbracket =-\llbracket x,z,y\rrbracket.
\end{equation*}
Further, if $x_{i},y_{i}\in \mathcal{B}$ are pairwise different, we have:
\begin{equation*}
D_{(x_{1},x_{2}),(y_{1},y_{2})}=-D_{(x_{1},x_{2}),(y_{2},y_{1})}=-D_{(x_{2},x_{1}),(y_{1},y_{2})}
\end{equation*}

In order to verify the $D_{x,y}$-identity, we will consider the following cases:
\begin{enumerate}
\item $x_{1}=x_{2}=y_{1}=y_{2}$;
\item only three elements among $\left\{x_{1},x_{2},y_{1},y_{2}\right\}$ are equal;
\item two pairs of elements among $\left\{x_{1},x_{2},y_{1},y_{2}\right\}$ are equal;
\item only two elements among $\left\{x_{1},x_{2},y_{1},y_{2}\right\}$ are equal;
\item all $\left\{x_{1},x_{2},y_{1},y_{2}\right\}$ are pairwise different.
\end{enumerate}

In the first case, $D_{(x_{1},x_{1}),(x_{1},x_{1})}=0$ and the $D_{x,y}$-identity trivially holds.

Admit that among $\left\{x_{1},x_{2},y_{1},y_{2}\right\}$ only three are equal. The possible cases are:
\begin{equation*}
\text{(i) }\ x_{1}=x_{2}=y_{1}, \ y_{2}\neq x_{1};\ \text{(ii) }\ x_{1}=x_{2}=y_{2},\ y_{2}\neq x_{1};\ \text{(iii) }\ x_{1}=y_{1}=y_{2},\ x_{2}\neq x_{1};\ \text{(iv) }\ x_{2}=y_{1}=y_{2},\ x_{1}\neq x_{2}.
\end{equation*}
It is clear that the last two subcases are a consequence of the first two (since $D_{x,y}=-D_{y,x}$). Further, with respect to the case (i) we have:
\begin{equation*}
D_{(x_{1},x_{1}),(x_{1},y_{2})}(z)=z\overline{x_{1}}x_{1}\overline{x_{1}}y_{2}-z\overline{x_{1}}y_{2}\overline{x_{1}}x_{1}=0,\text{ for all } z\in \mathcal{B}.
\end{equation*}
Thus, the $D_{x,y}$-identity trivially holds. The second case is analogous to the first.

Concerning the case (3), we may have three subcases:
\begin{equation*}
\text{(i) }\ x_{1}=x_{2},\ y_{1}=y_{2},\ x_{1}\neq y_{1};\ \text{(ii) }\ x_{1}=y_{1},\ x_{2}=y_{2};\ \text{(iii) }\ x_{1}=y_{2},\ x_{2}=y_{1}.
\end{equation*}
The second case is a trivial one. Concerning the other two subcases, we have $D_{x,y}=0$ by direct computations (using the previous lemma).

Let us now analyse the case (4). We can have six subcases:
\begin{equation*}
\text{(i) }\ x_{1}=x_{2};\ \text{(ii) }\ x_{1}=y_{1};\ \text{(iii) }\ x_{1}=y_{2};\ \text{(iv) }\ x_{2}=y_{1};\ \text{(v) }\ x_{2}=y_{2};\ \text{(vi) }\ y_{1}=y_{2},
\end{equation*}
(where, in each subcase, the remaining elements are pairwise different and different from the coincident ones). It is clear that not all subcases must be checked, due to the properties of the operators $D_{x,y}$. In fact, since $D_{x,y}=-D_{y,x}$, (i) implies (vi). Further, (iii), (iv) and (v) are a consequence of (ii), since:
\begin{equation*}
D_{(x_{1},x_{1}),(y_{1},x_{1})}=-D_{(x_{1},x_{1}),(x_{1},y_{1})},\ D_{(x_{1},x_{2}),(x_{2},y_{2})}=-D_{(x_{2},x_{1}),(x_{2},y_{2})}\ \text{ and } D_{(x_{1},x_{2}),(y_{1},x_{2})}=-D_{(x_{2},x_{1}),(y_{1},x_{2})},
\end{equation*}
respectively.
Concerning the subcase (i), for every $z\in \mathcal{B}$ we have:
\begin{eqnarray*}
D_{(x_{1},x_{1}),(y_{1},y_{2})}(z)  & = & \llbracket \llbracket z,x_{1},x_{1}\rrbracket , y_{1},y_{2} \rrbracket-\llbracket \llbracket z,y_{1},y_{2}\rrbracket ,x_{1},x_{1}  \rrbracket \\
 & = & z\overline{x_{1}}x_{1}\overline{y_{1}}y_{2}-z\overline{y_{1}}y_{2}\overline{x_{1}}x_{1}\\
 &=& z\overline{y_{1}}y_{2}-z\overline{y_{1}}y_{2}=0
\end{eqnarray*}
for all  $z\in \mathcal{B}$, concluding this case.

Concerning the subcase (ii), for every $z\in \mathcal{B}$ we have:
\begin{eqnarray*}
D_{(x_{1},x_{2}),(x_{1},y_{2})}(z)  & = & \llbracket \llbracket z,x_{1},x_{2}\rrbracket , x_{1},y_{2} \rrbracket-\llbracket \llbracket z,x_{1},y_{2}\rrbracket ,x_{1},x_{2}  \rrbracket \\
 & = & z\overline{x_{1}}x_{2}\overline{x_{1}}y_{2}-z\overline{x_{1}}y_{2}\overline{x_{1}}x_{2}\\
 &=& -z\overline{x_{2}}y_{2}+z\overline{y_{2}}x_{2}\\
 &=& -z\overline{x_{2}}y_{2}.
\end{eqnarray*}
Let us check both sides of the $D_{x,y}$-identity for $z_{1},z_{2},z_{3}\in \mathcal{B}$. We have:
\begin{eqnarray*}
LHS_{D}&=&D_{(x_{1},x_{2}),(x_{1},y_{2})}\left(\llbracket z_{1},z_{2},z_{3} \rrbracket\right)\\
&=& -2z_{1}\overline{z_{2}}z_{3}\overline{x_{2}}y_{2}.
\end{eqnarray*}
On the other hand, denoting by $RHS_{D}(1)$, $RHS_{D}(2)$ and $RHS_{D}(3)$, respectively, the three terms of the right side of the identity, we have:
\begin{equation*}
RHS_{D}(1)= \llbracket D_{(x_{1},x_{2}),(x_{1},y_{2})}(z_{1}) ,z_{2},z_{3}\rrbracket =  -2z_{1}\overline{x_{2}}y_{2}\overline{z_{2}}z_{3},
\end{equation*}
\begin{equation*}
RHS_{D}(2)= \llbracket  z_{1}, D_{(x_{1},x_{2}),(x_{1},y_{2})}(z_{2}),z_{3}\rrbracket =  -2z_{1}\overline{y_{2}}x_{2}\overline{z_{2}}z_{3},
\end{equation*}
and
\begin{equation*}
RHS_{D}(3)=  \llbracket  z_{1},z_{2},D_{(x_{1},x_{2}),(x_{1},y_{2})}(z_{3})\rrbracket =  -2z_{1}\overline{z_{2}}z_{3}\overline{x_{2}}y_{2}.
\end{equation*}
Now, since $RHS_{D}(3)=LHS_{D}$, we must verify if $RHS_{D}(1)+RHS_{D}(2)=0$. Using the properties of the previous lemma, it is possible to prove that $z_{1}\overline{y_{2}}x_{2}=-z_{1}\overline{x_{2}}y_{2}$. The proof of this fact must be divided in three cases: (i) $z_{1}=y_{2}$; (ii) $z_{1}=x_{2}$; (iii) $z_{1}$ different from $x_{2}$ and $y_{2}$. In the first two cases, we must use (4) of the previous lemma; in case (iii), we need to apply (5). Therefore, the two mentioned summands cancel. 

Finally, in the last case we have
\begin{equation*}
D_{x,y}(z)=z\overline{x_{1}}x_{2}\overline{y_{1}}y_{2}-z\overline{y_{1}}y_{2}\overline{x_{1}}x_{2}=0,\text{ for all } z\in  \mathcal{B}.
\end{equation*}

In order to prove the simplicity of this ternary algebra, let us consider $\mathbb{I}\neq \{0\}$ an ideal of $\mathcal{D}_{2}$. By definition, $\llbracket \mathbb{I},\mathcal{D}_{2},\mathcal{D}_{2} \rrbracket,\llbracket \mathcal{D}_{2},\mathbb{I},\mathcal{D}_{2} \rrbracket,\llbracket \mathcal{D}_{2},\mathcal{D}_{2},\mathbb{I} \rrbracket \subseteq \mathcal{D}_{2}$. Then, for for every $z\in \mathbb{I}\backslash \{0\}$ and for all $b_{i} \in  \mathcal{B}$, we have 
\begin{equation*}
 \llbracket  z,z,b_{i}\rrbracket =n(z)b_{i} \in \mathbb{I}\backslash \{0\} .
\end{equation*}
Thus, all basis elements belong to $\mathbb{I}$ and $\mathbb{I}=\mathcal{D}_{2}$,  ending the proof.
\end{proof}


\begin{Th}
Consider $\mathbb{H}( a,b )= \left\langle 1 \right\rangle _{\mathbb{F}} \oplus \mathbb{H}( a,b )_{s}$. 
Then $D\in Der(\mathcal{D}_{2})$ if and only if there exists $\Phi,\Psi\in End\left(\mathbb{H}( a,b )\right)$ such that $\Phi \in Der (\mathbb{H}( a,b ))$  and $\Psi(x)=x\Psi(1)$, for all $x\in \mathbb{H}( a,b )$ and $\Psi(1)\in \mathbb{H}( a,b )_{s}$ satisfying
\begin{equation*}
D=\Phi+\Psi.  \label{Thm_der_D_2}
\end{equation*}
\end{Th}

\begin{proof}
Admit that $D\in Der(\mathcal{D}_{2})$ . Then, from 
\begin{equation}
D\ \llbracket x,y,z\rrbracket  =\llbracket
D\left( x\right) ,y,z \rrbracket +\llbracket x,D\left(y\right) ,z\rrbracket +\llbracket x,y,D\left( z\right)\rrbracket \label{aux_D2_1}
\end{equation}
and from (\ref{CD-ternarymult}) we obtain
\begin{equation*}
D\ (xy) = D\ (x\overline{1}y)  = D(x)y+x\overline{D(1)}y+xD(y).
\end{equation*}
Setting $x=y=1$ in this identity we obtain
\begin{equation*}
D(1)+\overline{D(1)}=0,
\end{equation*}
which implies that $\overline{D(1)}=-{D(1)}$ and $D(1)\in \mathbb{H}( a,b )_{s}$. Whence, 
\begin{equation}
D\ (xy) = D(x)y-xD(1)y+xD(y). \label{aux_D2_2}
\end{equation}
Let us build $g\in End\left(\mathbb{H}( a,b )\right)$ such that
\begin{equation}
g(x)=D(x)-xD(1). \label{aux_D2_3}
\end{equation}
From (\ref{aux_D2_2}), it is possible to see that
\begin{equation*}
g(xy)=D(xy)-xyD(1)=D(x)y-xD(1)y+xD(y)-xyD(1)=g(x)y+xg(y),\ \text{for all}\ x,y\in \mathbb{H}( a,b ).
\end{equation*}
From here and from (\ref{aux_D2_3}) we have
\begin{equation*}
D(x)=g(x)+xD(1).
\end{equation*}

Reciprocally, admit that $D\in End\left(\mathbb{H}( a,b )\right) $ is such that
\begin{equation*}
D=\Phi+\Psi
\end{equation*}
where $\Phi,\Psi\in End\left(\mathbb{H}( a,b )\right)$ such that $\Phi \in Der \left(\mathbb{H}( a,b )\right)$  and $\Psi(x)=x\Psi(1)$, for all $x\in \mathbb{H}( a,b )_{s}$ and $\Psi(1)\in \mathbb{H}( a,b )_{s}$. Note that $\Psi(x)=x\Psi(1)$ also holds if $x\in \mathbb{H}( a,b )$. Then, for all $x,y,z\in \mathbb{H}( a,b )_{s}$ we have:
\begin{equation*}
\Psi\ \llbracket x,y,z\rrbracket  =\Psi\left( x\overline{y}z\right) =x\overline{y}z \ \Psi(1).
\end{equation*}
Further, 
\begin{eqnarray*}
\llbracket \Psi\left( x\right) ,y,z \rrbracket +\llbracket x,\Psi\left(y\right) ,z\rrbracket +\llbracket x,y,\Psi\left( z\right)\rrbracket & = &x\Psi\left( 1\right)\overline{y}z+x\overline{y\Psi\left( 1\right)}z+x\overline{y}z\Psi\left( 1\right)\\
 & = & x\Psi\left( 1\right)\overline{y}z-x\Psi\left( 1\right)\overline{y}z+x\overline{y}z\Psi\left( 1\right)\\ & = & x\overline{y}z \ \Psi(1).
\end{eqnarray*}
Note that this also holds if $x,y,z\in \left\langle 1 \right\rangle _{\mathbb{F}}$, so $\Psi \in Der(\mathcal{D}_{2})$. 

Concerning $\Phi\in  Der \left(\mathbb{H}( a,b )\right)$,  it is easy to conclude that $\Phi(1)=0$. Further, it can be shown that $\Phi \left(\mathbb{H}( a,b )\right)\subseteq \mathbb{H}( a,b )_{s}$, implying that, for any $y=\alpha \ 1 +y_{s}$, $y_{s}\in \mathbb{H}( a,b )_{s}$, we have
\begin{equation*}
 \overline{\Phi(y)} = \overline{\Phi\left(y_{s}\right)} =-\Phi \left( y_{s}\right) = \Phi \left( \overline{y}\right).
\end{equation*}
Thus, if $x,y,z\in \mathbb{H}( a,b )$, we have:
\begin{eqnarray*}
\Phi\ \llbracket x,y,z\rrbracket & =&\Phi\left( x\overline{y}z\right)\\
& = & \Phi (x) \overline{y}z+ x\Phi\left(\overline{y}\right)z + x\overline{y}\Phi\left( z\right)\\
   & = & \Phi (x) \overline{y}z+ x\overline{\Phi\left(y\right)}z + x\overline{y}\Phi\left( z\right)\\ 
  &=& \llbracket \Phi\left( x\right) ,y,z \rrbracket +\llbracket x,\Phi\left(y\right) ,z\rrbracket +\llbracket x,y,\Phi \left( z\right)\rrbracket .
\end{eqnarray*}
Whence, $\Phi\in Der(\mathcal{D}_{2})$ and the same happens with $D$.
\end{proof}

\begin{Lem}
All degree $1$ identities in $\mathcal{D}_{2}$ are a consequence of
\begin{equation*}
\llbracket y,x,x \rrbracket =\llbracket x,x,y\rrbracket.
\label{1-identity-D_2}
\end{equation*}
\end{Lem}
\medskip
\begin{Lem}
All degree $2$ identities in $\mathcal{D}_{2}$ are a consequence of
(\ref{1-identity-D_2})  and from the following degree $2$ identities:

\begin{equation*}
\llbracket \llbracket x,y,z\rrbracket ,u,v\rrbracket =\llbracket x,y,\llbracket
z,u,v\rrbracket \rrbracket,  \label{2-identity-1-D_2}
\end{equation*}
\begin{equation*}
\llbracket \llbracket x,y,z\rrbracket ,u,v\rrbracket =\llbracket x,\llbracket
u,z,y\rrbracket ,v\rrbracket.  \label{2-identity-2-D_2}
\end{equation*}
\end{Lem}

From the definition of the algebra $\mathbb{O}( a,b,c)$ it is possible to verify the following result.

\begin{Lem} $\mathcal{D}_{3}$ is not a ternary $D_{x,y}$-derivation algebra.
\end{Lem}

\begin{proof}
Consider an orhonormal basis $\mathcal{B}=\left\{1, a,b,ab,c,ac,bc,(ab)c\right\} $ of $\mathbb{O}( a,b,c)$, with the usual multiplication in this composition algebra. Let us take
\begin{equation*}
x_{1}=a=y_{1},\ x_{2}=b\ \text{and} \ y_{2}=c.
\end{equation*}
Then, 
\begin{equation*}
D=D_{x,y}(z)=\left(\left(\left(z\overline{a}\right)b\right)\overline{a}\right)c-\left(\left(\left(z\overline{a}\right)c\right)\overline{a}\right)b,\text{ for all } z\in \mathbb{O}( a,b,c ).
\end{equation*}
Let us take $z_{1}=ab$, $z_{2}=1$ and $z_{3}=c$. Then, $\llbracket z_{1},z_{2},z_{3}\rrbracket=\llbracket ab,1,c\rrbracket =(ab)c$ and  it is possible to obtain
\begin{equation*}
LHS_{D}=D\left((ab)c\right)=-2a.
\end{equation*}
On the other hand, since $D(ab)=-2ac$, $D(1)=2bc$ and $D(c)=2b$, it is possible to show that
\begin{equation*}
RHS_{D}=\llbracket D(ab),1,c\rrbracket+\llbracket ab,D(1),c\rrbracket+\llbracket ab,1,D(c)\rrbracket=2a.
\end{equation*}
Thus, $\mathbb{O}( a,b,c )$ is not a $D_{x,y}$-derivation algebra.
\end{proof}


\subsection{An analog of the TKK-construction for ternary algebras}

We recall the Tits-Kantor-Koecher (TKK for short) unified construction of the exceptional simple classical Lie algebras, by means of a composition algebra and a degree three simple Jordan algebra (see  \cite{Kantor}, \cite{Koecher} and \cite{Tits}). In this subsection we will use an analogue construction to define ternary multiplications and, if possible, ternary Jordan algebras.

Let $L=L_{-1}\oplus L_0\oplus L_1$ be a $3$-graded ternary algebra with the product $[x,y,z]$. By definition, we have:
\begin{equation*}
\left[L_{i},L_{j},L_{k}\right]\subseteq L_{i+j+k},
\end{equation*}
where the addition is considered modular (in ${-1,0,1}$). Following I. Kantor \cite{Kantor}, we define a ternary operation on $\mathcal{J}:=L_0$ by the rule:
\begin{equation}
 \llbracket x,y,z\rrbracket ={\mathcal{S}}_{x,y,z}[[[u_{-1},x,u_1],y,v_{-1}],z,v_1], \label{ternary_TKK}
\end{equation}
where ${\mathcal{S}}_{x,y,z}$ is the symmetrization operator in $x,y,z$ and  $u_i,v_i\in L_i$, $i=-1,1$.

Consider $L=A_1$ be the simple $4$-dimensional Filippov algebra over $\Co$ with the standard basis $\{e_1,e_2,e_3,e_4\}$ and the multiplication table
\begin{equation*}
 [e_1,\ldots,\hat{e_i},\ldots,e_4]=(-1)^ie_i. 
\end{equation*} 
Change this basis in $A_1$ to 
\begin{equation*}
a=\frac{{\bf i}}{2}e_1,\ b=\frac{1}{2}e_2,\ a_{-1}=e_3-{\bf i}e_4,\
 a_{1}=e_3+{\bf i}e_4, \text{ where } {\bf i}^2=-1.
\end{equation*} 
Then 
\begin{equation*}
\left\langle a_{-1}\right\rangle \oplus \left\langle a,b\right\rangle\oplus \left\langle a_{1}\right\rangle
\end{equation*}
is a $3$-grading on $A_1$, with $\mathcal{J}=L_{0}=\left\langle a,b\right\rangle$. Indeed, due to the anticommutativity of the multiplication $[.,.,.]$, to reach that conclusion it is enough to observe that 
\begin{equation*}
 [a,a_{-1},a_1]=-2b\quad \text{and}\quad [b,a_{-1},a_1]=-2a.
\end{equation*}
Putting $u_{-1}=v_{-1}=a_{-1},\ u_1=v_1=a_1$ in (\ref{ternary_TKK}), we obtain the following multiplication table in $\mathcal{J}$:
 \begin{equation}\label{tca1}
   \llbracket a,a,a\rrbracket = 6b,\quad \llbracket a,a,b\rrbracket = 2a,\quad \   \llbracket a,b,b\rrbracket =	-2b\quad \text{and}\quad \llbracket b,b,b\rrbracket =-6a.
 \end{equation}
 Then, we have
\begin{equation*}
   R_{(a,a)}=6e_{12}+2e_{21},\quad  R_{(a,b)}=2e_{11}-2e_{22}  \quad \text{and} \quad R_{(b,b)}=-2e_{12}-6e_{21},
\end{equation*}
and thus:
\begin{equation*}
  D_{(a,a),(a,b)}\doteq -3e_{12}+e_{21},\quad   D_{(a,a),(b,b)}\doteq e_{11}-e_{22}\quad \text{and} \quad D_{(a,b),(b,b)}\doteq -e_{12}+3e_{21}
\end{equation*}
where $\doteq$ denotes an equality up to a scalar and $e_{ij}$  is the matrix unit in the basis $\{a,b\}$. Now, we may consider a ternary commutative algebra $\mathcal{J}$ over an arbitrary field with the multiplication table (\ref{tca1}). The inclusion
\begin{equation*}
 D_{x,y}\in \left\langle -3e_{12}+e_{21},e_{11}-e_{22}, -e_{12}+3e_{21}\right\rangle 
\end{equation*} 
is immediate. Verifying the $D_{x,y}$-identity, we conclude that it holds if and only if $char \left(\mathbb{F}\right) = 2$.

In the Kantor article the product was defined on the space
 $L_{-1}$ by the rule $xy=[[a,x],y]$ for a fixed $a\in L_1$. We can try to
 do the same. Put
 \begin{equation*}
 \llbracket x,y,z\rrbracket={\mathcal S}_{x,y,z}[[[u_{0},x,u_1],y,v_{1}],z,v_0],
\end{equation*} 
 where $u_i,v_i\in L_i,\ i=0,1,\ x,y,z\in L_{-1}$. In this case we have 
\begin{equation*} 
 \llbracket a_{-1},a_{-1},a_{-1}\rrbracket =a_{-1},
 \end{equation*}
  with 
  \begin{equation*}
  a_{-1}=e_3-{\bf i}e_4,\ u_0=\frac{\textbf{i}}{4}e_1,\ v_0=e_2,\ u_1=v_1=a_1=e_3+{\bf i}e_4.
\end{equation*}
 It is easy to notice that every one-dimensional ternary algebra $\mathcal{J}$ is
 a ternary Jordan algebra, and $\mathcal{J}$ is simple if and only if $\{\mathcal{J},\mathcal{J},\mathcal{J}\}\neq
 0$.

\section{Reduced algebras of $n$-ary Jordan algebras}

Given an arbitrary class of $n$-ary algebras, $\mathcal{A}$, $n>2$,  with multiplication $\llbracket .,...,.\rrbracket$, let us fix $a\in \mathbb{V}$, the underlying vector space, and for each $i\in \{1,...,n\}$, define an $(n-1)$-ary algebras denoted by $\mathcal{A}_{i,a}$, by putting
\begin{equation*}
\llbracket x_{2},...,x_{n}\rrbracket _{i,a}=\llbracket x_{2},...,\underset{i\text{-th entry}}{\underbrace{ a
 }},...,x_{n}\rrbracket,
\text{ }x_{2},...,x_{n}\in \mathbb{V}.
\end{equation*}
Each algebra $\mathcal{A}_{i,a}$, defined on the same underlying space, is called a reduced algebra of $\mathcal{A}$. Under total commutativity or anticommutativity of $\llbracket .,...,.\rrbracket$, it is enough to consider $i=1$, which may be omitted by simply writing $\mathcal{A}_{a}$ and
\begin{equation*}
\llbracket x_{2},...,x_{n}\rrbracket _{a}=\llbracket a,x_{2},...,x_{n}\rrbracket,
\text{ }x_{2},...,x_{n}\in \mathbb{V}.
\end{equation*}

It may happen that each reduced algebra of an $n$-ary algebra belongs to the same class. Indeed, it is known that:
\begin{itemize}
\item reduced algebras of $n$-ary totally associative algebras are $(n-1)$-ary totally associative algebras;
\item reduced algebras of $n$-ary totally (anti)commutative algebras are $(n-1)$-ary totally (anti)commutative algebras;
\item reduced algebras of $n$-ary Leibniz algebras are $(n-1)$-ary Leibniz algebras;
\item reduced algebras of $n$-ary Filippov algebras algebras are $(n-1)$-ary Filippov algebras \cite{Fil};
\item reduced algebras of $n$-ary Malcev algebras are $(n-1)$-ary Malcev algebras \cite{Pojidaev}.
\end{itemize}
So, it is natural to put the following question: are the reduced algebras of $n$-ary Jordan algebras $(n-1)$-ary Jordan algebras? 

Consider $\mathcal{A}=\left( \mathbb{V},\llbracket .,...,.\rrbracket \right) $ an $n$-ary Jordan
algebra ($n> 2$) and let us fix $a$ on an arbitrary basis of $\mathbb{V}$. Define
now an $(n-1)$-ary algebra $\mathcal{A}_{a}=\left( \mathbb{V},\llbracket .,...,.\rrbracket
_{a}\right) $ such that
\begin{equation*}
\llbracket x_{2},...,x_{n}\rrbracket _{a}=\llbracket a,x_{2},...,x_{n}\rrbracket ,
\text{ }x_{2},...,x_{n}\in \mathbb{V}.
\end{equation*}
In order to analyze whether or not $\mathcal{A}_{a}$  is an $\left( n-1\right) $-ary Jordan algebra, observe that a right multiplication
operator $R_{x}$ in $\mathcal{A}_{a}$, where $x=\left( x_{3},...,x_{n}\right) $, is
defined as follows:
\begin{equation*}
zR_{x}=\llbracket z,x_{3},...,x_{n}\rrbracket _{a}=\llbracket
a,z,x_{3},...,x_{n}\rrbracket =\llbracket z,a,x_{3},...,x_{n}\rrbracket ,
\end{equation*}
which can be written in terms of a right multiplication operator in $\mathcal{A}$,
since
\begin{equation*}
R_{x}=R_{\left(a,x_{3},...,x_{n}\right)}=R_{\overline{x}},
\end{equation*}
where $\overline{x}=\left( a,x_{3},...,x_{n}\right)$.

Now, the commutator of right multiplications in $\mathcal{A}_{a}$ is given by:
\begin{eqnarray*}
D_{x,y}(z) &=&z\left( R_{x}R_{y}-R_{y}R_{x}\right) =D_{\overline{x},%
\overline{y}}(z) \\
&=&\llbracket \llbracket z,a,x_{3},...,x_{n}\rrbracket ,a,y_{3},...,y_{n}\rrbracket
-\llbracket \llbracket z,a,y_{3},...,y_{n}\rrbracket ,a,x_{3},...,x_{n}\rrbracket .
\end{eqnarray*}
Since $D_{\overline{x},\overline{y}}\in Der\left( \mathcal{A}\right) $, we have:

\begin{eqnarray*}
D_{x,y}\ \llbracket z_{2},z_{3},...,z_{n}\rrbracket _{a} &=&D_{\overline{x},%
\overline{y}}\ \llbracket a,z_{2},z_{3},...,z_{n}\rrbracket  \\
&=& \llbracket D_{\overline{%
x},\overline{y}}\left( a\right) ,z_{2},z_{3},...,z_{n}\rrbracket
+\sum_{i=2}^{n}\llbracket a,z_{2},...,D_{\overline{x},\overline{y}}\left(
z_{i}\right) ,...,z_{n}\rrbracket\\
&=&\llbracket \llbracket \llbracket a,a,x_{3},...,x_{n}\rrbracket
,a,y_{3},...,y_{n}\rrbracket -\llbracket \llbracket a,a,y_{3},...,y_{n}\rrbracket
,a,x_{3},...,x_{n}\rrbracket ,z_{2},z_{3},...,z_{n}\rrbracket
\end{eqnarray*}
\begin{eqnarray}
\label{red}
&&+\sum_{i=2}^{n}\llbracket z_{2},...,D_{x,y}\left( z_{i}\right)
,...,z_{n}\rrbracket _{a}. 
\end{eqnarray}

It happens that  the first summand in the last development may be different from
zero and thus $\mathcal{A}_{a}$ may not be an $\left( n-1\right) $-ary Jordan algebra.

\begin{Th}
The reduced algebras of the ternary Jordan algebra $\mathbb{A}$ defined in the third
section are not Jordan algebras.
\end{Th}
\begin{proof}
A counterexample can be observed by taking the above mentioned first summand, and considering, \textit{e.g.}, $n=3$, \ $\dim \mathbb{V}=4$, and putting $a=b_{1}$, $x_{3}=b_{2}$ and  $y_{3}=b_{1}$ ($b_{1}$ and $b_{2}$ in an orthonormal basis). Then 
\begin{equation*}
 \llbracket \llbracket a,a,x_{3}\rrbracket,a,y_{3}\rrbracket -\llbracket \llbracket a,a,y_{3}\rrbracket,a,x_{3}\rrbracket=-2b_{1},
\end{equation*}
and it is clear that that summand may not be zero.
\end{proof}

\begin{Remark}
Since the ternary multiplication of a Jordan triple system (recall remark (\ref{other})) is only partially commutative, it is straightforward that its reduced algebras may not be Jordan algebras. 
\end{Remark}

\begin{Remark}
The main subclass of $n$-ary Jordan algebras consists of totally commutative and totally associative $n$-ary algebras.
As  follows from (\ref{red}), the reduced algebras of any totally commutative and totally associative $n$-ary algebras are totally commutative and totally associative $(n-1)$-ary algebras.
\end{Remark}

\end{document}